\documentclass[a4paper,11pt]{amsart}

\usepackage{amsmath,amsthm,amsfonts,amssymb,amscd,yfonts}
\numberwithin{table}{section}
\numberwithin{equation}{section}
\usepackage{commath}
\usepackage{mathtools}
\usepackage{verbatim}
\usepackage[usenames,dvipsnames,usenames,dvipsnames,svgnames,table]{xcolor}
\usepackage{cite,mdwlist}

\usepackage[unicode,bookmarks, pdftex, backref = page]{hyperref}
\definecolor{newblue}{RGB}{0,102,204}
\definecolor{newred}{RGB}{206,32,41}
\hypersetup{colorlinks=true,citecolor=newblue,linkcolor=newred,
  urlcolor=magenta,
pdfpagemode=UseNone, breaklinks=true}

\usepackage[english]{babel}
\usepackage{latexsym}
\usepackage{amsopn}
\usepackage{mathrsfs}
\usepackage[active]{srcltx}
\usepackage{fullpage}
\usepackage{xfrac}
\usepackage{tabulary}
\usepackage{resizegather}
\usepackage{fouriernc}

\usepackage{enumitem}

\usepackage{mathtools}
\mathtoolsset{showonlyrefs=true}

\theoremstyle{plain}
\newtheorem{theorem}{Theorem}[section]
\newtheorem{corollary}[theorem]{Corollary}
\newtheorem{proposition}[theorem]{Proposition}
\newtheorem{lemma}[theorem]{Lemma}

\theoremstyle{definition}
\newtheorem{notation}{Notation}[section]
\newtheorem{definition}{Definition}[section]

\theoremstyle{remark}
\newtheorem{remark}{Remark}[section]

\title{An explicit bound for the log-canonical 
  degree of curves on open surfaces.}
\author{Pietro Sabatino}
\address{Via Val Sillaro 5 - 00141 Roma}
\email{pietrsabat@gmail.com}
\thanks{ The author is really grateful to M.~McQuillan for his continuous
  support, encouragement and uncountable discussions and suggestions. 
  This project started during the \emph{Junior
  Trimester Program - Algebraic Geometry} (February-April 2014)
  at the \emph{Hausdorff 
  Research Institute for Mathematics} (Bonn), to the institute its staff and
  participants of the trimester goes the gratitude of the author 
  for the hospitality
  and the fruitful scientific atmosphere. 
   The author would also like to thank F.~Polizzi and 
  A.~Rapagnetta for useful discussions and suggestions. Finally the author
  would like to express his gratitude to the anonymous reviewer whose valuable
  comments improved the clarity and quality of this paper. 
  }

\begin{document}

\begin{abstract}
  Let $X$, $D$ be a smooth projective surface and a simple normal
  crossing divisor on $X$, respectively. 
  Suppose $\kappa (X, K_X + D)\ge 0$,
  let $C$ be an irreducible curve on $X$
  whose support is not contained in $D$
  and $\alpha$ a rational number in $ [ 0, 1 ]$.
  Following Miyaoka, we define an orbibundle
  $\mathcal{E}_\alpha$ as a suitable free subsheaf of log differentials on a
  Galois cover of $X$. 
  Making use of $\mathcal{E}_\alpha$ we prove a Bogomolov-Miyaoka-Yau inequality for
  the couple $(X, D+\alpha C)$.
  Suppose moreover that $K_X+D$ is big and nef and $(K_X+D)^2 $ is greater than
  $e_{X\setminus D}$, namely the topological Euler number of the open surface
  $X\setminus D$. 
  As a consequence of the inequality, 
  by varying $\alpha$, we deduce a bound for
  $(K_X+D)\cdot C$ by an explicit function of the invariants: $(K_X+D)^2$, 
  $e_{X\setminus D}$ and $e_{C \setminus D} $, namely the
  topological Euler number of the normalization of $C$ minus
  the points in the set theoretic counterimage of $D$. We finally deduce that
  on such  surfaces curves with $- e_{C\setminus D}$ bounded form a bounded
  family, in particular there are only a finite number of curves $C$ on $X$
  such that $- e_{C\setminus D}\le 0$. 
\end{abstract}

\maketitle

\section{Introduction and statement of results}
Let $X$ be a minimal complex projective surface of general type 
such that $K_X^2 > c_2(X)$,
in \cite{bogomolov-families}, Bogomolov proved the well known result 
according to which irreducible
curves of fixed geometric genus on $X$ form a bounded family.
Since Bogomolov's argument depended on
the analysis of curves contained in a certain closed set (see
\cite{deschamps-courbes} for an exposition),
his remarkable result was not effective.
Indeed, Bogomolov was able to prove that curves in this closed set form a 
bounded family by considerations involving
algebraic foliations but without providing an explicit bound on their degree. 
Because of this, in a
deformation of the surface $X$, the number of either rational or elliptic curves
might in principle tend to infinity. 
This situation can be ruled out providing 
an upper bound on the canonical degree of irreducible curves on $X$ by a function 
of the invariants of $X$ and the geometric genus of the curve. 
The existence of such a function and its form was then conjectured
in various places and
in slightly different contexts, see for instance
\cite[\S 9]{tian-ke_algebraic_manifolds}, with the function depending only on
$K^2_X$, $c_2(X)$ and the geometric genus of the curve. 
The conjecture was
proved with some restrictive hypothesis on the singularities of the curve
involved by Langer in \cite{langer-orbifold} and finally in its full generality
by Miyaoka in \cite{miyaoka-orbibundle}. It is interesting to note that part of
Miyaoka's result can be recovered by methods closer in spirit to the original
argument of Bogomolov, see McQuillan \cite[Corollary 1.3]{mcquillan-mixed}, 
though one is able to prove the existence of the afore mentioned function no
explicit form can be established. 

The aim of the present paper is to prove a bound as in \cite{miyaoka-orbibundle}
but in the contest of open surfaces. 
Let then $X$ be a smooth projective surface, $C$ be an irreducible curve and $D$ a 
simple normal crossing divisor on $X$. 
In what follows we will assume that the curve $C$ is not part
of the boundary divisor $D$, even if not explicitly stated.  
Regarding divisors and line bundles we will generally follow the terminology and 
notation of \cite{lazarsfeld-positivityI}, in particular a simple normal
crossing divisor is a reduced divisor whose components are smooth and cross
normally. Moreover, given a divisor $D$ we identify its support,
$\mathrm{Supp}\ (D)$ with the underlying effective reduced divisor. 

First of all, following Miyaoka \cite{miyaoka-orbibundle}, we are going to prove a
Bogomolov-Miyaoka-Yau inequality for the couple $(X, D+\alpha C)$, $\alpha \in
[0,1]$ a rational number, and then 
deduce our bound from this inequality. In order to state the result we need a 
couple of definitions. 

\begin{definition} \label{notation:relativepoincare}
  Denote by $e_{X\setminus D}$ the topological Euler
  number of the open surface $X\setminus D$. Note that $e_{X\setminus D} =
  e_{\mathrm{orb}}(X\setminus D)$ is the orbifold Euler number of $(X,D)$, see Remark
  \ref{remark:explicitep}.
  Let $\eta \colon \widetilde{C}\rightarrow C$ be the normalization of
  $C$, we set
  \begin{equation*}
	  e_{C\setminus D} := e_{\mathrm{top}}\left( \widetilde{C}\setminus
	  \eta^{-1}(D)\right) =  e_{\mathrm{top}}( \widetilde{C} )- 
    \sharp \left( \eta^{-1}(D)\right) \ ,
  \end{equation*}
  namely the topological Euler number of the open set $\widetilde{C}\setminus
  \nu^{-1}(D)$.
\end{definition}

\begin{definition}
  \label{definition:Drational}
  Let $C$ be a curve on $X$ not contained in $D$,
  we say that $C$ is a \emph{ smooth $D$-rational curve }
  if $C\cong \mathbb{P}^1$ and $D \cdot C\le 1$. In other words, $C$ is smooth,
  it crosses $D$ transversally and $C\setminus D$ contains an open set
  isomorphic to $\mathbb{A}^1$.
\end{definition}

\begin{theorem}[Bogomolov-Miyaoka-Yau type inequality]
  \label{theorem:main}
  Let $X$ be a smooth projective surface, $D$ be a simple normal crossing divisor on 
  $X$ and $C$ an irreducible curve on $X$ not contained in $D$. 
  Suppose that $K_X+D$ is a $\mathbb{Q}$-effective divisor \footnote{In other words
    the log Kodaira dimension of $X\setminus D$ is greater than or equal to
    zero, namely $\kappa(X, K_X + D ) \ge 0$.}, 
  then:  
  \begin{enumerate} [label = (\roman*),ref = {\thetheorem .(\roman*)}]
    \item \label{item:maininequality}
      If $\alpha$ is a real number $\alpha\in [0,1]$, then the following
      inequality holds:
      \begin{equation} \label{equation:maininequality}
	\frac{\alpha^2}{2}  \left[ C^2 + 3(K_X+D) \cdot C + 
	  3 e_{ C \setminus D} 
	\right] 
	-2 \alpha \left[ (K_X+D)\cdot C  
	+ \frac{3}{2} e_{ C \setminus D} \right]
	+ 3 e_{X\setminus D} - (K_X+D)^2\ge 0\ .
      \end{equation}
    \item \label{item:mainsecondinequality} Suppose moreover that $C$ is
      not a smooth $D$-rational curve and $( K_X + D )\cdot C \ge -
      \tfrac{3}{2} e_{X\setminus D}$ then the following inequality
      holds:
      \begin{equation}
	\label{equation:mainsecondary}
	2 \left[ (K_X+D)\cdot C + \dfrac{3}{2} 
	e_{C \setminus D} \right]^2  - 
	\left[ 3 e_{X\setminus D} - (K_X+D)^2 \right] 
	\left[ C^2 + 3 (K_X+D) \cdot C + 3 e_{C\setminus D} 
	\right] \le 0\ .
      \end{equation}
  \end{enumerate}
\end{theorem}

\begin{remark}
  Observe that in case $D$ is the zero divisor then $ -\frac{1}{2} e_{C\setminus
  D} = g-1$, where as usual $g$ denotes the geometric genus of $C$. It is clear
  then that in this case Theorem \ref{item:maininequality} is a generalization
  of \cite[Theorem 1.3 (i), (ii)]{miyaoka-orbibundle}, in which 
  $ -\frac{1}{2} e_{C\setminus D} $ plays the role of $g-1$.  
  In case $\alpha = 0$,
  Theorem \ref{item:maininequality} coincides with 
  \cite[Theorem (7.6)]{sakai-semistable}.
\end{remark}

It is worth noting that Theorem \ref{item:maininequality} is not a direct
consequence of a general Bogomolov-Miyaoka-Yau inequality in
the form of 
\cite[Theorem 0.1]{langer-orbifold}. Indeed we do not impose any 
restriction on the
singularities of the curve $C$, hence $(X, D+\alpha C)$ may not be log
canonical. Though Theorem \ref{theorem:main} suits our present
needs, it seems then natural to ask whether or not it is possible to prove,
by a modification of the argument provided here, a version of Theorem
\ref{item:maininequality} for a more general couple $(X, B)$, where
$B=\sum_{i}\beta_i B_i$, $\beta_i\in [0,1]$ rational,
and the couple may have worse singularities then log canonical. 
We will address this question in a successive paper. 

Given Theorem \ref{theorem:main} a direct argument will lead us to: 

\begin{theorem}
  \label{theorem:bound}
  Let $X$ be a smooth projective surface, $D$ be a simple normal crossing divisor on 
  $X$ and $C$ an irreducible curve on $X$ not contained in $D$. 
  Suppose moreover that $K_X + D$ is $\mathbb{Q}$-effective. 
  \begin{enumerate}[label = (\roman*),ref = {\thetheorem .(\roman*)}]
    \item \label{item:uppercanonicalbound} 
      If $K_X+D$ is nef, $(K_X + D)^2>0$, $(K_X + D
      )^2 > e_{X\setminus D}$ and moreover $C$ is not a smooth $D$-rational
      curve then the
      relative canonical degree of $C$ is bounded by:
      \begin{equation}\label{equation:uppercanonicalbound}
	(K_X+D)\cdot C \le A \left(- \dfrac{1}{2} e_{C\setminus D} \right) + B\ , 
      \end{equation}
      where $A$, $B$ depends only on $( K_X + D )^2$,
      $e_{X\setminus D}$ and can be chosen as:
      \begin{gather*}
	A = \dfrac{2 (K_X + D)^2 + 
	  \sqrt{2 (K_X + D)^2 \left( 3e_{X\setminus D} - \left( K_X + D
	\right)^2 \right)}}{ \left( K_X + D \right)^2 - 
	e_{X\setminus D} } \ , \\
	B = \dfrac{(K_X + D)^2 
	  \left( 3e_{X\setminus D} - (K_X + D)^2 \right) 
	  + 2 e_{X\setminus D} \sqrt{ 2 (K_X + D)^2 
	    \left( 3e_{X\setminus D} - (K_X + D)^2 
	\right) }}{2 \left( 
	    (K_X + D)^2 - e_{X \setminus D}
	\right)}  \ .
      \end{gather*} 
    \item \label{item:uppercanonicalboundsmooth} 
      If $K_X+D$ is nef, $C$ is smooth 
      and $D$ and $C$ intersects transversally  
      then the relative canonical degree of $C$ is bounded by:
      \begin{gather}
	\label{equation:canonicalsmooth}
	(K_X +D) \cdot C \le 
	- \frac{3}{2} e_{C\setminus D} + 
	\dfrac{
	  \sqrt{3 e_{X\setminus D} - (K_X + D)^2 } 
	  \sqrt{-2 e_{C\setminus D}  + 3e_{X\setminus D} - 
	(K_X + D)^2 }}{2} + \\
	\dfrac{3e_{X\setminus D} - (K_X + D)^2 }{2} \ ,
      \end{gather}
      if $C$ is not a smooth $D$-rational curve. If
      $C$ is a smooth $D$-rational curve, namely it is 
      isomorphic to $ \mathbb{P}^1$ and $D.C \le 1$, 
      then the relative canonical degree of $C$ is bounded by:
      \begin{equation}\label{MichaelsEqn}
	(K_X+D)\cdot C \le   3e_{X\setminus D} - (K_X+D)^2 -3. 
      \end{equation}
  \end{enumerate}
\end{theorem}

Theorem \ref{item:uppercanonicalbound} corresponds to \cite[Theorem
1.1]{miyaoka-orbibundle} and Theorem \ref{item:uppercanonicalboundsmooth}
corresponds to \cite[Corollary 1.4]{miyaoka-orbibundle}.
In analogy with this last result, 
for a smooth curve $C$ that meets $D$ transversally and such that 
$-e_{C\setminus D}$ is very large, the relative canonical degree is bounded by a
function that asymptotically behaves like $\frac{3}{2} \left(- e_{C\setminus
D}\right)$.
As remarked in \cite{mcquillan-mixed}, considerations of differential geometric
nature suggest that, in such bounds on the canonical degree, 
good choices for the constant in front
of $ e_{C\setminus D}$ are the reciprocal of 
either $- \frac{2}{3}$ or $- \frac{1}{2}$, the  
holomorphic sectional curvature of the K\"ahler-Einstein metric of balls and
bi-discs, respectively. It turns out that in the algebraic geometric setting,  
see \cite{gasbarri-canonicaldeg}, $- \frac{1}{2}$ is optimal, in particular
taking into account singular curves. For smooth curves 
$- \frac{2}{3} $ seems the right choice, at least asymptotically. 

\begin{corollary}[Uniform bound of the relative canonical degree]
  \label{corollary:bound}
  Let $X$ be a smooth projective surface, $D$ be a simple normal 
  crossing divisor on 
  $X$ and $C$ an irreducible curve on $X$ not contained in $D$. 
  If $ K_X + D$
  is nef and big and moreover $( K_X + D )^2>
  e_{X\setminus D}$ then the relative canonical degree of $C$ is
  bounded by
  \begin{equation*}
    ( K_X + D )\cdot C \le A \left(- \dfrac{1}{2} e_{C \setminus D} 
    \right) + B
  \end{equation*}
  where $A$,
  $B$ depend only on $e_{X\setminus D}$ and $( K_X + D )^2$.
\end{corollary}

By the above bound on the canonical degree
it follows that curves for which $ - e_{C\setminus D} $ is
fixed form a bounded family. In particular: 

\begin{corollary} \label{corollaryp:boundedfamily}
  Let $X$ be a smooth projective surface and $D$ a simple normal
  crossing divisor on $X$. 
  Suppose that $ K_X + D$
  is nef and big and moreover that $( K_X + D )^2>
  e_{X\setminus D}$. Then, curves $C$ on $X$ that are not contained in $D$ and
  such that $- e_{C\setminus D}$ is
  bounded form a bounded family, where the 
  number of components is bounded by a function that
  depends on $(K_X + D)^2$ and
  $e_{X\setminus D}$. In particular on $X$ there are only a finite
  number of curves $C$ such that $-e_{C\setminus D}\le 0$ and their 
  number is bounded by a function of $(K_X + D)^2$ and
  $e_{X\setminus D}$.
\end{corollary}

\begin{remark} 
  In view of the hypotheses of Corollary \ref{corollaryp:boundedfamily}, a bound
  on the log canonical degree of curves translates in an analogous bound on
  their
  degree with respect to any fixed ample divisor. The same bound, given its
  nature, holds uniformly in a family of deformations of the surface $X$ too. It is
  worth noting that since surfaces of log general type with $(K_X+D)^2$ bounded
  are bounded, see \cite[Theorem 7.7]{alexeev-boundedness} for instance, hence 
  the conclusions of Corollary \ref{corollaryp:boundedfamily} hold in the most
  general sense. 
\end{remark}

\begin{remark}
  Let $\widetilde{C}$ be the normalization of $C$, $\eta\colon
  \widetilde{C} \rightarrow C$ the corresponding map and consider
  $\eta^{-1}(D)$ as a closed set. By definition of
  $e_{C\setminus D}$, if $-e_{C\setminus D}\le 0$ then the geometric genus of
  $C$ is less than or equal to
  one and there are only four possibilities 
  for the open set $\widetilde{C} \setminus \eta^{-1}(D)$, namely
  \begin{equation*}
    \widetilde{C} \setminus \eta^{-1}(D) \cong \begin{cases}
      \mathbb{P}^1 \\
      \mathbb{A}^1 \\
      \mathbb{A}^1\setminus \left\{ pt \right\}  \\
      \textrm{Elliptic curve}
    \end{cases}
  \end{equation*}
  hence Corollary \ref{corollaryp:boundedfamily} generalizes
  \cite[Corollary 1.2]{miyaoka-orbibundle}.
\end{remark}

We end the series of results with the following Corollary, where we apply
Theorem \ref{item:uppercanonicalbound} to an elementary (i.e. that can be
formulated in elementary terms) situation in $
\mathbb{P}^2$. 

\begin{corollary} \label{corollary:p1example}
  Let $D_1$, $D_2$, $C$ be distinct irreducible curves in $ \mathbb{P}^2$
  of degree
  $d_1$, $d_2$ and $d$, respectively, such that $D = D_1+D_2$ is a
  simple normal crossing divisor.  
  Denote by $g$ the geometric genus of $C$, consider $\widetilde{C}$ the
  normalization of $C$, $\eta\colon
  \widetilde{C} \rightarrow \mathbb{P}^2$ the induced map and define 
  \begin{equation*}
    m := \min_{ p \in \widetilde{C} \cap \eta^{-1}(D)} 
    \left\{ \mathrm{mult}_{p}\left(
    \eta^*D \right) \right\}\ .
  \end{equation*}
  Suppose that $d\ge d_2 \ge d_1 > 0$ and set
  \begin{equation*}
    \lambda = \frac{d_1}{d_2}\ , \ \nu= \frac{d}{d_2}\ . 
  \end{equation*}
  There exist constants $\lambda_0$, $ \frac{2}{3} < \lambda_0 < 1 $,
  $\mathrm{h}$ and $\mathrm{k}$ such that 
  if $d_2\ge 6$, $\lambda \ge \lambda_0$ and 
  \begin{equation}
    \nu >  \frac{\mathrm{h} g+\mathrm{k}}{\left( \frac{\lambda_0}{2} - 
      \frac{1}{3} \right) \left(
    \frac{\lambda_0}{2} + \frac{1}{4} \right) } 
  \end{equation}
  then 
  \begin{equation}
    m \le \left \lfloor \dfrac{50}{
    \left( \frac{\lambda}{2} - \frac{1}{3} \right)}\right \rfloor \ .
  \end{equation}
\end{corollary}

The statement of Corollary \ref{corollary:p1example} may result a bit obscure 
at a first reading, but basically its content is the following.
After arranging the degrees of $D$ in such a way that the
hypotheses of Theorem \ref{item:uppercanonicalbound} are satisfied, if the
degree of $C$ is sufficiently large, then the order of tangency between $C$ and
$D$ can not be everywhere too high. Note moreover that we need $D$ to have at
least two components to guarantee the required flexibility in order to
obtain a couple
$(X,D)$ of log general type, minimal and such that $(K_{ \mathbb{P}^2 } + D)^2
- e_{\mathbb{P}^2\setminus D} > 0$. 
In contrast to the elementary nature of
the statement we are not aware of any elementary proof. Details of the
proof are provided at the end of \S \ref{section:final}.

\section{Preliminaries}
In the present section, for reader's convenience, we gather a number of
results that we will use during the course of our proofs. For the sake of
clarity we state them in the form more suitable to our needs. Where it is
possible without complicating the discussion, we provide short proofs, if they
are not
available elsewhere or if they provide a way to quickly gain insight on the
particular topic.  

\subsection{Zariski decomposition with support in a negative cycle}
The Zariski decomposition was introduced in \cite[\S 7]{zariski-rr}
and its proof involves a rather elementary although lengthy argument in linear
algebra and quadratic forms. Following Miyaoka, 
see \cite[\S 2]{miyaoka-orbibundle}, 
our argument will rely on a slight
modification of the classical Zariski decomposition. Basically we
require the support of the negative part to be contained in a fixed
negative definite cycle. 
The existence of the Zariski decomposition with support can be proved 
following, with minor modifications, the original argument of Zariski. 
Zariski constructs the negative part of the decomposition
and then as a consequence the positive part, but given the fact that the positive
part can be interpreted as a solution to a maximization problem (this was
already remarked for instance by Kawamata in \cite[Proposition (1.5) and
(1.6)]{kawamata-noncompletesurfaces}), 
it turns out that it is much easier
to start by constructing the positive part and then deduce the existence of the
Zariski decomposition, see for instance
\cite{bauer-simple_zariski} and \cite{bauer-ssimple_zariski} for an even
more elementary exposition. Following this approach we
are going to summarize results regarding
Zariski decomposition with support and its relation with the classical one. 
It is worth noting that a similar discussion is contained in 
\cite{laface-zariski_support}, nonetheless we prefer to briefly summarize it here in a
way more convenient for our needs.

\begin{notation}
  We denote by $\preccurlyeq$ the partial order on
  $\mathrm{Div}_{\mathbb{R}}\left( X\right) $ given by $D_1 \preccurlyeq
  D_2$ if $D_2 - D_1$ is effective.
\end{notation}

\begin{definition}
  Let $E_1,\ldots,E_l$ be irreducible curves on $X$, the cycle
  $E=\sum_{i=1}^{l} E_i$ is said to be \emph{negative definite} if the
  intersection matrix relative to $E$, $\left( E_i\cdot E_j \right)_{ij}$, is
  negative definite \footnote{Note that if $E$ is negative definite then
  the curves $E_i$ must be distinct.}. In order to simplify the exposition
  we will consider the trivial cycle negative definite. 
\end{definition}

\begin{proposition} \label{proposition:zariskidecomposition}
  Let $D\in \mathrm{Div}_{\mathbb{Q}}\left( X \right)$ be effective and $E =
  \sum_{i=1}^{l} E_i $ be a negative definite cycle
  then there exist 
  $P_E(D),\ N_E(D) \in  \mathrm{Div}_{\mathbb{Q}}\left( X \right)$ such that:
  \begin{enumerate}[label = (\roman*),ref = {\thetheorem .(\roman*)}]
    \item \label{item:zariskidec1} $P_E(D)$ and $N_E(D)$ are effective and 
      $D = P_E(D) + N_E(D)$.
    \item $P_E(D)$ is nef on $E$ namely $P_E(D)\cdot E_i\ge 0 $ for every
      $i=1,\ldots,l$.
    \item The support of $N_E(D)$ is contained in $E$, namely $N_E(D) =
      \sum_{i=1}^{l} a_i E_i$, $a_i\ge 0$ for $i=1,\dots , l$.
    \item $P_E(D)$ is numerically trivial on $N_E(D)$, namely $P_E(D)\cdot
      E_i = 0$
      for
      every prime component $E_i$ in the support of $N_E(D)$. It follows that
      $P_E(D) \cdot N_E(D) = 0$ and then $D^2 = P_E^2(D) + N^2_E(D)$.
    \item \label{item:zariskidec5} If the above properties are satisfied then
	    $P_E(D)$ and $N_E(D)$ are unique. In particular,
      $P_E(D)$ can be characterized as the largest effective 
      $\mathbb{Q}$-divisor such that
      $P_E(D) \preccurlyeq D$ and $P_E(D)$ is nef on $E$ (namely if $P^\prime$
      is an effective $Q$-divisor $ P^\prime \preccurlyeq D $ and
      $P^\prime$ is nef on
      $E$ then $P^\prime \preccurlyeq P_E(D) $). 
  \end{enumerate}
\end{proposition}

\begin{proof}
	First of all, 
	we are going to prove that there exist unique $\mathbb{R}$-divisors
  $P_E(D)$ and $N_E(D)$ that
  satisfy properties \emph{(i)}--\emph{(iv)}, that these divisors 
  are rational will 
  follow immediately from unicity, namely they will be the unique solution of a
  system of linear equations with integral coefficients.
  Up to reordering, we may suppose that $E_1,\ldots,E_k$, $k\le l$, are the
  components of $E$ contained in the support of $D$. Put 
  \begin{equation*}
    D = \sum_{j=1}^{h+k} d_j D_j
  \end{equation*}
  where the irreducible components of the support of $D$ not contained
  in $E$ are denoted by $D_1,..., D_h$, moreover 
  $D_{h+1}=E_1,\ldots,D_{h+k} = E_{k}$, 
  and we set $E^\prime =
  \sum_{i=1}^k E_i$. Consider the linear space of $\mathbb{R}$-divisors
  with support in $D$, we can write its elements as $\sum_{j=1}^{h+k} x_j
  D_j$. In this linear space consider the compact subset defined by:
  \begin{equation}
    \label{equation:initialxj}
    0\le x_j \le d_j,\quad j=1,\ldots,h+k
  \end{equation}
  and
  \begin{equation}
    \label{equation:initialDE}
    \sum_{j=1}^{h+k}x_j \left( D_j\cdot E_i\right)\ge  0\quad i=1,\ldots,k\ .
  \end{equation}
  This compact set contains at least one point that maximizes the
  function $\sum_{j=1}^{h+k} x_j $, let $P_E(D)$ be the corresponding divisor,
  $N_E(D) = D - P_E(D)$, \emph{(i)} is then satisfied. Moreover since $P_E(D)$ 
  satisfies \eqref{equation:initialDE}, it is nef on $E^\prime$ and then on $E$.
  Observe that in \eqref{equation:initialDE} the coefficient 
  $D_j\cdot E_i$ can be negative only for $j> h$, hence these inequalities do
  not impose any restrictions on $x_j$ for $j\le h$. It follows that without 
  loss of generality we can substitute 
  the first $h$ inequalities in \eqref{equation:initialxj} with the equalities
  \begin{equation}
    \label{equation:xj}
      x_j=d_j, \quad j=1,\ldots,h
  \end{equation}
  and then $N_E(D)$ satisfies \emph{(iii)}. 
  Since $P_E(D)$ maximizes $\sum_j x_j$, for
  every fixed $1\le i \le k$ such that $d_{h+i}\ne 0$ and small 
  $\epsilon > 0 $, $P_E(D)+\epsilon E_i$ is not
  nef on $E^\prime$, and then 
  \begin{equation}
    \label{equation:limitepsilon}
    \left( P_E(D) + \epsilon E_i \right) \cdot E_i <0
  \end{equation}
  since the other intersections $\left( P_E(D) + \epsilon E_i \right)\cdot E_j$,
  $j\ne i$, are still non negative. Passing to the limit in
  \eqref{equation:limitepsilon} we get $P_E(D)\cdot E_i\le 0$ but since $P_E(D)$
  is nef on $E^\prime$ we have $P_E(D)\cdot E_i=0$ and this
  concludes the proof of \emph{(iv)}.

  Observe that in view of \emph{(iv)}
  we can now
  rewrite \eqref{equation:initialDE} as
  \begin{equation} \label{equation:DE}
    \sum_{j=1}^{h+k}x_j \left( D_j\cdot E_i\right)=  0\quad i=1,\ldots,k\ ,
  \end{equation}
  and hence $P_E(D)$ is a solution of the $h+k$ equations given by
  \eqref{equation:xj}, \eqref{equation:DE}. If we write this system of
  equations in matrix form we obtain:
  \begin{equation*}
    \begin{bmatrix}
      I_h & 0 \\
      A & \bf{E'}
    \end{bmatrix}\cdot \bf{X} =
    \begin{bmatrix}
      \bf{d} \\
      \bf{0}
    \end{bmatrix}
  \end{equation*}
  where $I_h$ is the $h\times h$ identity matrix, $\bf{E'}$ is the
  $k\times k$ negative definite intersection matrix of $E^\prime$ and $A$ is a
  $k\times h$ matrix with rational integral entries. Since
  all involved coefficients are rational, then the unique solution of the
  above system of equations corresponding to $P_E(D)$ is rational.    

  For the proof of the last part of \emph{(v)} we are going to use 
  the following Lemma.
  \begin{lemma}\label{lemma:maxdivisors}
    Let $P$, $P^\prime$ be two effective divisors $P, P^\prime \preccurlyeq
    D$, $P = \sum_{j=1}^{h+k} y_j D_j$ and $P^\prime =
    \sum_{j=1}^{h+k} y^\prime_{j} D_j$. If $P$, $P^\prime$ are both nef on $E$
    then $\max \left( P, P^\prime \right) := \sum_{j=1}^{h+k} \max\left(
    y_j, y^\prime_j\right) D_j$ is nef on $E$.
  \end{lemma}
  \begin{proof}[Proof of Lemma \ref{lemma:maxdivisors}]
	  Given an element in the linear space of $\mathbb{R}$-divisors with
	  support contained in $D$, it is nef on $E$ if and only if its
	  coordinates $x_1,\dots,x_{h+k}$ satisfy the inequalities in
	  \eqref{equation:initialDE}. Consider then the $i$-th inequality in
	  \eqref{equation:initialDE}, its left-hand side is a linear polynomial
	  in the $x$'s that has only one negative coefficient, namely
	  $D_{h+i}\cdot E_i$. We may suppose without loss of
	  generality that $y_{h+i}\ge y^\prime_{h+i}$, then 
	  \begin{equation*}
		  \left( \max \left( P, P^\prime \right) - P \right) \cdot E_i
		  \ge 0
	  \end{equation*} 
	  from which it follows that $ \max \left( P, P^\prime
	  \right)\cdot E_i \ge 0$.  Since the above argument holds for every
	  $i=1,\ldots,k$, this concludes the proof of the Lemma.
  \end{proof}
  
  Let us now complete the proof of \emph{(v)}. 
  Let $P^\prime_E(D) \preccurlyeq D  $ be an effective
  divisor 
  that is nef on $E$. Since the support of the negative part is contained
  in $E^\prime$ (see \eqref{equation:xj}) 
  then there exist $x_i \ge 0$, $i=1,\ldots,k$, such that 
  \begin{equation*}
    \max \left( P_E(D), P_E^\prime(D) \right) = 
    P_E(D) + \sum_{i=1}^{k} x_i E_i\ .
  \end{equation*}
  Since by Lemma \ref{lemma:maxdivisors} $\max \left( P_E(D),
  P^\prime_E(D)\right) $ is nef on $E^\prime$ then for $j=1,\ldots,k$ we have
  \begin{equation*}
    \sum_{i=1}^{k} x_i E_i\cdot E_j \ge 0
  \end{equation*} 
  and then
  \begin{equation*}
    \left( \sum_{i} x_i E_i \right)\cdot \left( \sum_i x_i E_i \right) = \sum_j
    \left( \sum_i x_i E_i \cdot E_j \right) \ge 0 \ .
  \end{equation*}
  By hypothesis the intersection matrix of $E$, and then that of
  $E^\prime$, is negative definite, the above inequality implies $x_i=0$,
  $i=1,\dots,k$ and
  then $\max \left( P_E(D), P_E^\prime(D) \right) = P_E(D) $. This concludes the
  proof of \emph{(v)} and of the Proposition.
\end{proof}

\begin{definition}
  Given an effective $Q$-divisor $D$ and a negative definite cycle $E$ we
  will call the decomposition $D = P_E(D) + N_E(D)$ of Proposition
  \ref{proposition:zariskidecomposition} the Zariski
  decomposition of $D$ with support in $E$. In particular we will call
  $P_E(D)$ the $E$-nef part of $D$ and $N_E(D)$ the $E$-negative part of $D$. 
\end{definition}

\begin{remark} \label{remark:supportinclu}
  From the proof of Proposition \ref{proposition:zariskidecomposition} it
  follows that the Zariski decomposition of a divisor $D$ with support on a
  particular negative cycle $E$ depends only on the part of $E$ supported 
  in $N_E(D)$.  
\end{remark}

\begin{remark}
  In order to distinguish the Zariski decomposition with support from the
  classical one, we will refer to the latter
  as the \emph{absolute Zariski decomposition}
  and denote it by $D =
  P\left(D \right) + N\left( D \right)$. 
  The reasons behind this terminology
  as well as the relation between the two decompositions will be clarified
  by the next Corollary. In particular, the absolute Zariski
  decomposition coincides with the Zariski decomposition of $D$ with
  support in $N \left( D \right)$, moreover among the various 
  Zariski
  decompositions of $D$ the absolute one is characterized by having
  maximal negative part and minimal nef part. 
\end{remark}

\begin{corollary} \label{corollary:zdecompositionpr}
    Let $D$, $D^\prime$ two effective $Q$-divisors and $E$,
    $\widehat{E}$ negative definite cycles.
    \begin{enumerate} [label = (\roman*) ]
      \item If $D\preccurlyeq D^\prime$ then $P_E\left(D \right) \preccurlyeq
	P_E\left( D^\prime \right) $.
      \item If $E \preccurlyeq \widehat{E}$ then $N_E\left(D \right)
	\preccurlyeq N_{\widehat{E}}\left( D \right)$, $P_E\left(D \right)
	\succcurlyeq P_{\widehat{E}}\left( D \right)$, $0\ge \left(
	N_E\left(D \right) \right)^2
        \ge \left( N_{\widehat{E}}\left( D \right) \right)^2 $ and 
	$\left( P_E\left(D \right) \right)^2
        \le \left( P_{\widehat{E}}\left( D \right) \right)^2 $.
      \item Let $\mathrm{N}$ be the support of $N\left( D \right)$,
	$\mathrm{N}$
	is a negative definite cycle, $N\left( D \right) =
	N_{\mathrm{N}}\left( D \right)$, 
	$P\left( D \right) =
	P_{\mathrm{N}}\left( D \right)$, 
	$N_E\left(D \right)
	\preccurlyeq N \left( D \right)$, $P_E\left(D \right)
	\succcurlyeq P \left( D \right)$, $0\ge \left(
	N_E \left( D \right) \right)^2
	\ge \left( N \left( D \right) \right)^2 $ and 
	$\left( P_E\left(D \right) \right)^2
	\le \left( P \left( D \right) \right)^2 $.
    \end{enumerate}
\end{corollary}

\begin{proof}
  Observe that $P_E\left( D \right) \preccurlyeq D \preccurlyeq
  D^\prime$ and $P_E\left( D \right) $ is nef on $E$, then \emph{(i)}
  follows by
  Proposition \ref{item:zariskidec5}. 
  Similarly, in
  \emph{(ii)},
  $P_{\widehat{E}}\left( D \right)\preccurlyeq D$, it is nef on 
  $\widehat{E}$ and 
  then on $E$. By Proposition \ref{item:zariskidec5}, 
  $P_E\left(D \right) \succcurlyeq P_{\widehat{E}}\left( D \right)$ and
  consequently
  \begin{equation*}
    N_E\left(D \right) = D - P_E\left(D \right)
    \preccurlyeq D - P_{\widehat{E}}\left( D \right) =
    N_{\widehat{E}}\left( D \right)\ .
  \end{equation*}
  Let us prove the inequalities involving intersection numbers in
  \emph{(ii)}. In
  the vector space of cycles whose support is contained in $D$ 
  denote by $V$ and $\widehat{V}$ cycles whose support
  are contained in $E$, $\widehat{E}$, respectively. 
  Observe that since $V \subseteq
  \widehat{V}$ then $V^{\perp} \supseteq \widehat{V}^{\perp}$ and we have
  the two orthogonal decompositions:
  \begin{gather*}
    D = P_E\left( D \right) + N_{E}\left( D \right) \in V^\perp \oplus V \\
    D = P_{\widehat{E}}\left( D \right) + N_{\widehat{E}}\left( D \right) \in 
    \widehat{V}^\perp \oplus \widehat{V}\ .
  \end{gather*}
  It follows that 
  \begin{equation*}
    R = P_E\left( D \right) - P_{\widehat{E}}\left( D \right) =
    N_{\widehat{E}}\left( D \right) - N_{E}\left( D \right)\in V^{\perp}
    \cap \widehat{V}
  \end{equation*}
  and we have the two orthogonal decompositions
  \begin{gather*}
    P_E\left( D \right) = P_{\widehat{E}}\left( D \right) + R \in
    \widehat{V}^{\perp} \oplus \widehat{V} \\
    N_{\widehat{E}}\left( D \right) =  N_{E}\left( D \right) + R \in
    V\oplus V^{\perp}
  \end{gather*}
  from which the inequalities involving intersection numbers in
  \emph{(ii)} follow
  directly by observing that since $R \in \widehat{V}$ then $R^2\le 0$. 

  Finally by definition of absolute Zariski decomposition of $D$, it coincides
  with the Zariski decomposition with support contained in the negative
  part $N\left( D \right)$. 
  Since $P\left( D \right) $ is nef by Proposition
 \ref{item:zariskidec5} we have $\overline{P}\left( D
  \right) \preccurlyeq P_E\left(D \right)$, it follows that  
  $N_E\left(D \right) \preccurlyeq N \left( D \right)$.
  Without loss of generality, see Remark
  \ref{remark:supportinclu}, we can assume that the support of $E$ is
  contained in $\mathrm{N}$, the remaining inequalities in \emph{(iii)}
  then follow by part \emph{(ii)} of the Corollary. 
\end{proof}

\begin{corollary}
  Let $D_1,\ D_2\in \mathrm{Div}_{\mathbb{Q}}\left( X \right)$ be
  effective divisors and let $E$ be a negative definite cycle. 
  \begin{enumerate} [label = (\roman*) ]
    \item If $D_1 \equiv_{num.} D_2$ then $N_E\left( D_1 \right) =
      N_E\left( D_2 \right)$ and $N\left( D_1 \right) = N\left( D_2
      \right)$.
    \item If moreover $D_1 \equiv_{lin.} D_2$ then $P_E\left( D_1 \right)
    \equiv_{lin.}  P_E\left( D_2 \right) $ and $P\left( D_1 \right) \equiv_{lin.}
    P\left( D_2 \right)$. 
  \end{enumerate}
  \label{corollary:zariskilinear}
\end{corollary}

\begin{proof}
  The second statement follows directly from the first one. Let us prove
  the first one for the absolute Zariski decomposition, the case of
  relative Zariski decomposition can be proved by an analogous argument. 
  Observe that, by the maximality of the positive part, a curve $F^\prime$
  is contained in the support of the negative part of $D_i$ if and only if
  $D_i\cdot F^\prime< 0$. It follows that $N\left( D_1 \right)$ and $ N \left(
  D_2 \right)$ are supported on the same negative definite cycle, say
  $F = \sum_{j=1}^{k} F_j$. Let $\bf{F} $ be the $k\times k$ negative
  definite intersection matrix of $F$ and write $N\left( D_i \right) =
  \sum_{j=1}^{k}x_j F_j $, then ${\bf X}^i = \left( x^i_1, \ldots,x^i_j,\ldots,
  x^i_k \right)$ is the unique solution to the system of linear equations 
    \begin{equation*}
      {\bf F} \cdot {\bf X}^i = {\bf d}_i
    \end{equation*}
    where ${\bf d}_i = \left( D_i\ F_1, \ldots, D_i\ F_k \right)$, since 
    ${\bf d}_1 = {\bf d}_2$, this concludes the proof.  
\end{proof}

\begin{remark}
  \label{remark:zariskigeneral}
  In view of Corollary \ref{corollary:zariskilinear}, it does make sense to 
  consider the Zariski decomposition, either relative or absolute, of a
  divisor that is linearly equivalent to an effective $\mathbb{Q}$-divisor.
\end{remark}

\subsection{Nef reduction}
Let $\rho \colon Z\rightarrow Y$ be a surjective morphism between
non singular projective surfaces whose exceptional locus $R$ is a
divisor of simple normal crossings. Suppose moreover 
\begin{equation*}
  R \subseteq \rho^{-1}\left( \Lambda \right)\subseteq \Delta\ , 
\end{equation*}
where $\Delta$ is an normal crossing divisor on $Z$ and
$\Lambda$ an effective reduced divisor on $Y$. 
Observe that $\rho \left( R \right)$ is a finite
set and  then we
can find an affine open subset $U$ containing $\rho ( R )$ in
which $\Lambda$ is defined by a single equation, say $\lambda$.
Hence $\rho^*\left( \dif \log\ \lambda \right) $ is a
section of $\Omega^1_{Z}\left( \log \Delta \right)_{|\rho^{-1}(U)}$
defined around $R$. Indeed it suffices to check this statement 
locally at each
point $q \in R \subseteq \rho^{-1}\left( \Lambda \right)$. 
Indeed since
$\rho^{-1}\left( \Lambda \right)$ is contained in $\Delta$ 
of simple normal crossing, there exist $z_1, z_2$
local coordinates around $q$ such that 
\begin{equation*}
  \rho^{*} \lambda = u\ z_1^a z_2^b,
\end{equation*}
$a,b \ge 0$, $a+b> 0$, $u$ a unit, and then
\begin{equation} \label{equation:pullbackdlog}
  \rho^*\left( \dif \log\ \lambda \right) = a \dfrac{\dif z_1}{z_1} + b
  \dfrac{\dif z_2}{z_2} + \textrm{(regular $1$-form)}\ , 
\end{equation}
with $(a,b)\ne (0,0)$. 

\begin{proposition}
  \label{proposition:nefreduction}
  Let $\rho \colon Z\rightarrow Y$, $R$, $\Delta$, $\Lambda$, $\lambda$
  as above and $\mathcal{E}$ a rank two vector bundle on $Z$ such that:
  \begin{enumerate}
    \item $\mathcal{E}$ is a subsheaf of $\Omega^1_{Z}( \log \Delta ) $.
    \item $D = c_1 ( \mathcal{E} )$ is a $\mathbb{Q}$-effective divisor.
    \item $N_R( D )$ is an integral divisor.
    \item There exist a neighborhood $V$ of $R$ such that $\rho^*(
	    \dif \log\ \lambda ) \in \Gamma (V, \mathcal{E}) \subseteq
	    \Gamma( \Omega^1_{Z}(V, \log \Delta
       )) $. 
  \end{enumerate}
  Then, there exist a rank two vector bundle $\mathfrak{P}_R(
  \mathcal{E}) \subset \mathcal{E}$ such that:
  \begin{enumerate}
    \item $ \mathfrak{P}_R(
     \mathcal{E} )  = \mathcal{E}$ outside of $R$;
    \item $ c_1 ( \mathfrak{P}_R (
      \mathcal{E} ) ) =
      P_R ( D )$;
    \item $c_2 ( \mathfrak{P}_R (
      \mathcal{E} ) ) = c_2 (
      \mathcal{E} ) $.
  \end{enumerate}
\end{proposition}

\begin{remark}
  Proposition \ref{proposition:nefreduction} is basically
  \cite{miyaoka-orbibundle} Lemma 2.3. Though the hypothesis in Miyaoka are
  slightly different from ours, indeed he supposes that the divisor $E$
  coincides with the exceptional locus of the map $\rho$, his proof still
  works without any substantial modification here as well.
\end{remark}

For reader's convenience, we briefly recall: 

\begin{proof}[Proof of Proposition \ref{proposition:nefreduction}]
  Since $\rho^*(\dif \log \lambda ) $  is a nowhere vanishing section of
  $\mathcal{E}$ in $V$, it induces an injection $\mathcal{O}_V\rightarrow
  \mathcal{E}_{\vert V}$. In view of \eqref{equation:pullbackdlog}, after
  restricting $V$ if needed, the cokernel of the above injection is locally
  free of rank one and then isomorphic to $\det ( \mathcal{E}
  )_{\vert V}$, summing up we get an exact sequence 
  \begin{equation*}
    0 \rightarrow \mathcal{O}_V \rightarrow \mathcal{E}_{\vert V} \rightarrow
    \mathcal{O}_V(D)\rightarrow 0\ .
  \end{equation*}
  In order to simplify notation set $N= N_R(D)$ and $P=P_R(D)$. With a slight
  abuse of notation we continue 
  to denote by $N$ the subscheme relative to the sheaf of ideals 
  $\mathcal{O}_Z(-N)$. 
  Observe that by hypothesis $N$ is a subscheme of $V$,
  consider the map obtained composing the surjections 
  \begin{equation*}
    \mathcal{E} \rightarrow \mathcal{E}_{\vert V}\rightarrow
    \mathcal{O}_V(D) \rightarrow
    \mathcal{O}_N(D)
  \end{equation*}
  and denote by $\mathfrak{P}_R(\mathcal{E})$ its kernel. 
  Since $P$ is numerically
  trivial on $N$, we get the exact sequence 
  \begin{equation*}
    0 \rightarrow \mathfrak{P}_R(\mathcal{E}) \rightarrow \mathcal{E} \rightarrow
    \mathcal{O}_N(N) \rightarrow 0
  \end{equation*}
  then
  \begin{equation*}
    c\left( \mathfrak{P}_R(\mathcal{E}) \right) (1+ N) = 
    c(\mathcal{E})   
  \end{equation*}
  from this equality
  \begin{equation*}
    c_1\left( \mathfrak{P}_R(\mathcal{E}) \right) = P
  \end{equation*}
  and finally
  \begin{equation*}
    c_2( \mathfrak{P}_R(\mathcal{E})) = c_2(\mathcal{E} ) - c_1(
    \mathfrak{P}_R(\mathcal{E})) N =  c_2(\mathcal{E} ) - P N =
    c_2(\mathcal{E} )\ .
  \end{equation*}
\end{proof}

\begin{definition}
  We call $ \mathfrak{P}_R( \mathcal{E}) $, 
  as in the Proposition \ref{proposition:nefreduction}, 
  a \emph{$R$-nef reduction} of $\mathcal{E}$.
\end{definition}

\subsection{Bogomolov-Miyaoka-Yau type inequality for subsheaves of
logarithmic forms}
In order to obtain our explicit version of Bogomolov-Miyaoka-Yau
inequality, following a quite common path (see for
	instance, \cite{megyesi-bmy}, \cite{langer-bmy},
\cite{miyaoka-orbibundle}, \cite{langer-orbifold}, only to name a few) 
we will start with a generalized inequality applied to a
suitable subsheaf of logarithmic forms. 
The proof will then consist in a
straightforward computation of the Chern classes appearing in the
inequality and a careful estimation of the contributions due to the
singularities involved. In particular we will make use of the following:

\begin{theorem}
  \label{theorem:bmy}
  Let $X$ be a smooth projective surface, $D$ a simple 
  normal crossing
  divisor on $X$ and $\mathcal{E}$ a locally free rank two subsheaf of
  $\Omega^1_{X}\left( \log D \right)$ such that $c_1 ( \mathcal{E} )$ 
  is an effective $\mathbb{Q}$-divisor. Then
  \begin{displaymath}
    3 c_2( \mathcal{E} ) + \dfrac{1}{4}  \left[ N \left( c_1 (
    \mathcal{E} ) \right) \right]^2 \ge c^2_1\left( \mathcal{E} \right) .
  \end{displaymath}
\end{theorem}

\begin{proof}
  See Miyaoka \cite[Remark 4.18, p.~170]{miyaoka-maximal}, \cite[Theorem
  0.1]{langer-bmy}.
\end{proof}

\begin{remark}
  \label{remark:bmysheaves}
  We did not state Theorem \ref{theorem:bmy} in the most general version
  available,
  indeed the above form will suffice for most of our purposes.  
  We will need Bogomolov-Miyaoka-Yau
  inequality in the generality provided by \cite[Theorem 0.1]{langer-orbifold}
  only 
  to prove 
  \eqref{MichaelsEqn}.
  For a generalization to the case of log canonical surfaces see for instance
  \cite{langer-orbifold}.
\end{remark}

\section{Main construction} \label{section:mainconstruction}
In this Section, and for the rest of this paper, 
we denote by $X$ a smooth projective surface,
$D$ a simple normal crossing divisor on $X$ 
and $C$ an irreducible curve on $X$ not contained in
$D$. Let $\alpha \in [0,1]\cap \mathbb{Q}$, following
Miyaoka \cite{miyaoka-orbibundle} \S 3, after taking a log resolution of
$C+D$, we are going to construct a vector bundle $\mathcal{E}_\alpha$ 
in terms of $X$, $C$, $D$ and $\alpha$. Namely, given
a log resolution of $C+D$, say $\widetilde{Y}$,
we consider logarithmic forms 
\footnote{For general facts on logarithmic forms used 
here the reader may consult 
\cite{viehweg-lectures}.
}
with poles on the total transform 
$\overline{C+D}$ and, roughly speaking, define $\mathcal{E}_\alpha$ as the
kernel of the map
\begin{equation*}
  \rho: \Omega^1_{\widetilde{Y}}\left( \log \overline{C+D} \right)
  \rightarrow \mathcal{O}_{(1-\alpha) \widetilde{C}} 
\end{equation*}
induced by the natural residue map
\begin{equation*}
  \Omega^1_{\widetilde{Y}}\left( \log \overline{C+D} \right)\rightarrow
  \mathcal{O}_{\widetilde{C}}\rightarrow 0
\end{equation*}
where $\widetilde{C}$ denotes the strict transform of $C$. Since $1-\alpha$
is not in general integral, in order to formulate this definition 
in a rigorous way we have implicitly to consider roots of a local
equation of $\widetilde{C}$.  
We will then take a suitable ``\emph{Kawamata covering}''
$f: Z \rightarrow \widetilde{Y}$ (see \cite[Theorem
17]{kawamata-characterizationav}, \cite[Proposition
4.1.12]{lazarsfeld-positivityI})
in which these roots are defined, or in other words, such that
$(1-\alpha)f^*\widetilde{C}$ is an integral divisor. $\mathcal{E}_\alpha$
will be defined on the Galois covering $Z$, by construction it will inherit an
equivariant action and it will then be, technically speaking, an
``\emph{orbibundle}'' on $\widetilde{Y}$.

We are going to detail the construction of $\mathcal{E}_\alpha$, since
in the following Section we will perform explicit computations involving its
Chern classes, we are going to take some time expressing them in terms of 
log resolution data.

\subsection{Definition of \texorpdfstring{$\mathcal{E}_\alpha$}{Ea}} 
\label{subsection:notation}
For reader convenience we will adopt, as far as possible,
Miyaoka's notation in \cite{miyaoka-orbibundle}, the main difference being
that we need to distinguish carefully between singular points contained in the
boundary divisor $D$ and points that are not.
First of all, consider the set of points where $C+D$ fails to be simple normal
crossing, say $S$. We can express $S$ as a disjoint union $S= S_1
\dot\bigcup S_2$, where $S_1$ is the set of singular points of $C$ that are
not contained in $D$ and $S_2$ is the set of points of $C\cap D$ where
$D+C$ is not simple normal crossing. We are going to consider a log-resolution
of $D+C$.
Let us start blowing up $X$ with centers in the points of $S$ say $\mu\colon
Y\rightarrow X$, moreover observe that we can write $\mu=\mu_2 \circ
\mu_1$, where $\mu_1\colon Y^\prime\rightarrow X$ is the blowing up of $X$ with
centers in the points of $S_1$ and $\mu_2\colon Y \rightarrow Y^\prime$ the 
blowing up of $Y^\prime$ with centers in the set $\mu_1^{-1}(S_2)$.
Let us denote by $s^\prime, s$ the cardinality of the sets $S_1$ and $S$, 
respectively. Let $E_1,\dots , E_s$ be the exceptional curves of $\mu$,
where $E_1,\dots,E_{s^\prime}$ come from $\mu_1$ and 
$E_{s^\prime+1},\dots, E_{s}$ come from $\mu_2$.
Since $\mathrm{Supp}\ \mu^{*}(C+D)$ is not necessarily simple normal 
crossing, we proceed
until we get
$\pi\colon \widetilde{Y} \rightarrow Y$ given by a composition 
\begin{equation*}
  \widetilde{Y}=Y_r\xrightarrow{\pi_r} Y_{r-1} \xrightarrow{\pi_{r-1}}\cdots
  \xrightarrow{\pi_{r^\prime+1}} Y_{r^\prime}\xrightarrow{\pi_{r^\prime
   }}Y_{r^\prime-1}\xrightarrow{\pi_{r^\prime - 1 }} Y_{r^\prime -
  2}\rightarrow \cdots 
  \xrightarrow{\pi_1} Y_0 = Y\ ,
\end{equation*}
where
\begin{itemize}
	\item At each intermediate 
		step, $\pi_i$ is the blowing up of $Y_{i-1}$
		in a point where 
		$\mathrm{Supp}\ \pi_{i-1}^{*}\cdots
		\pi_1^{*}\mu^{*}(C+D)$ 
		is not simple
		normal crossing. 
	\item The first $r^\prime$  blow-ups are centered in
		points mapping to $S_1$.
	\item The remaining $(r-r^\prime)$ blow-ups are centered in
		points mapping to $S_2$.
	\item In the end, denote by $\pi = \pi_r \circ \cdots \circ \pi_1$,
		$\mathrm{Supp}\ \pi^{*}\mu^{*}(C+D)$
		is a simple normal crossing divisor in $\widetilde{Y}$ that
		we denote 
		by
		$\overline{C+D}$ 
		\footnote{A word of caution is in order here.
			Following Miyaoka,
		$\overline{C+D}$ does not denote the total transform of the
	curve $C+D$, and indeed it is a reduced divisor. On the other hand,
	$\overline{E}_{i}$ denotes the total transform in $\widetilde{Y}$ of
the exceptional curve $E_i$.}.
\end{itemize}
We will, moreover, use the following additional notation for the elements of the
above resolution.

{\small \setlength\extrarowheight{6pt}
\centering 
\begin{tabulary}{1.0\textwidth}{ C |  C }
    \textbf{Notation} & \textbf{Definition} \\
    $E_{s+i}$, $i=1,\ldots,r^\prime$ & Exceptional curve of the 
    blow-up $\pi_i\colon Y_i\rightarrow Y_{i-1}$,  
    $i=1,\ldots,r^\prime$ centered in
    points mapping to $S_1$. \\
    $E_{s+i}$, $i=r^\prime + 1,\ldots,r$ & Exceptional curve of the blow-up
    $\pi_i\colon Y_i\rightarrow Y_{i-1}$,  $i=r^\prime + 1,\ldots,r$
    centered in points mapping to $S_2$. \\
    $F_1,\ldots,
    F_{s^\prime},F_{s^\prime+1},\ldots,
    F_s \textrm{Div}(\widetilde Y)$ & Strict transforms in
    $\widetilde{Y}$ of the exceptional curves $E_1,\dots, E_s$, respectively.
    \\   
    $G_{1},\ldots, G_{r^\prime},
    G_{r^\prime + 1},\ldots , G_{r}\in \textrm{Div}(\widetilde Y)$ & 
    Strict transforms in
    $\widetilde{Y}$ of the exceptional curves $E_{s+1},\dots, E_{s+r}$, respectively.
    \\
    $\overline{E}_1,\ldots,
    \overline{E}_{s^\prime},\overline{E}_{s^\prime+1},\ldots, \overline{E}_s,$
    $\overline{E}_{s+1},\ldots, \overline{E}_{s+r^\prime},
    \overline{E}_{s+r^\prime + 1},\ldots , \overline{E}_{s+r}\in
    \textrm{Div}(\widetilde Y)$
    & 
    Total transforms in
    $\widetilde{Y}$ of the exceptional curves $E_{1},\dots, E_{s+r}$, respectively.
    Note that for the last blow-up
    $\overline{E}_{s+r} =
    E_{s+r} = G_{s+r}$.
    \\
        $F^\prime$ & $F^\prime = F_1+\cdots +F_{s^\prime}$ \\
    $F^{\prime \prime}$ & $F^{\prime \prime} = F_{s^\prime+1}+ \cdots + F_s$ \\ 
    $F$ &  $F = F^\prime + F^{\prime\prime}$ \\
    $G^\prime$ & $G^\prime = G_{1}+\cdots+ G_{r^\prime}$ \\
    $G^{\prime \prime}$ & $G^{\prime \prime} = G_{r^\prime + 1}+\cdots+ G_{r}$ \\
    $G$ & $G = G^\prime + G^{\prime \prime}$ \\
    $E^\prime$ & $E^\prime = F^\prime + G^\prime$ \\
    $E^{\prime \prime}$  & 
    $E^{\prime \prime} = F^{\prime \prime}+ G^{\prime \prime}$ \\
    $E$ & $E = E^\prime + E^{\prime \prime} =F+G$
\end{tabulary}
}

The exceptional locus of $\pi\circ \mu$ is equal to
$F + G$ and 
\begin{equation*}
  \overline{C+D}=\widetilde{C} + \widetilde{D} + F + G
\end{equation*}
where $\widetilde C, \widetilde D$ denote the strict
transforms of $C$ and $D$, respectively.
Given $\alpha \in [0,1]\cap \mathbb{Q}$, we can consider, see
Proposition 4.1.12 in \cite{lazarsfeld-positivityI}, a \emph{``Kawamata
covering''} of the couple 
$ \widetilde Y $, $\overline{C+D}$, namely
a finite covering $f \colon Z\rightarrow \widetilde Y$ such that
$ A_\alpha = (1-\alpha)f^* \widetilde C \preccurlyeq f^*\widetilde C $ is an 
integral divisor, $\mathrm{Supp}\ f^*\widetilde{C}$ is smooth and
$\mathrm{Supp}\ f^*
\overline{C+D}$ has simple normal crossings.
The natural residue map induces a 
surjection $$f^*\Omega^1_{\widetilde Y}
(\log \overline{C+D})\longrightarrow
{\mathcal{O}_{Z}}/{\mathcal{O}_{Z}(-A_\alpha)}
\rightarrow 0$$ define $ \mathcal{E}_\alpha $ as the kernel of this
map so that it fits in the exact sequence
\begin{equation}\label{equation:Ealphaexact}
  0 \rightarrow \mathcal{E}_\alpha \longrightarrow 
  f^*\Omega^1_{\widetilde Y}(\log \overline{C+D})\longrightarrow
  {\mathcal{O}_{Z}}/{\mathcal{O}_{Z}(-A_\alpha)}
  \rightarrow 0\ ,
\end{equation}
it follows from the preceding exact sequence that 
$\mathcal{E}_\alpha \subseteq f^*\Omega^1_{\widetilde Y}(\log
\overline{C+D})$ is locally free and of rank two. 

\begin{remark} \label{remark:Ealphaaspullback}
	In defining $\mathcal{E}_\alpha$ we have followed closely
	\cite{miyaoka-orbibundle}, but it is worth 
	noting that the end result is equal to
	$\widehat{f^*\Omega}^1_{\widetilde Y} \left(\log (\alpha
	\widetilde{C}+\widetilde{D}+E)\right)$, according to Langer's
	notation in \cite[\S 2, p.~365]{langer-orbifold}. This can be 
	checked in a rather straightforward manner by a computation 
	in local coordinates. To give the flavor, consider a point $p\in
	\mathrm{Supp}\ f^*\widetilde{C}$ that is also a smooth point of
	$\mathrm{Supp}\ (\alpha \widetilde{C}+ \widetilde{D} + E ) $. Then
	there exist local coordinates, $z^\prime$ around $f(p)$ such that
	$\widetilde{C}$ has equation $z^\prime =0$ and $z$ around $p$ such that
	$\mathrm{Supp}\ f^*\widetilde{C}$ has equation
	$z=0$. If $t=0$ is a local equation for $f^*(\alpha \widetilde{C})$
	around $p$ then
	we may suppose without loss of generality that $t=z^{\alpha m}$, where
	the integer $m> 0$ is such that $\alpha m$ is again integral. By
	definition of $\widehat{f^*\Omega}^1_{\widetilde Y} \left(\log (\alpha
	\widetilde{C}+\widetilde{D}+E)\right)$ and $\mathcal{E}_\alpha$ then:
	\begin{equation*}
		\begin{split}
			\widehat{f^*\Omega}^1_{\widetilde Y} & \left(\log (\alpha 
			\widetilde{C}+\widetilde{D}+E)\right) =
			\dfrac{f^* \dif z^\prime}{t}
			\mathcal{O}_{\widetilde{Y}} + f^*
			\Omega^1_{\widetilde{Y}} =
			\dfrac{ \dif f^* z^\prime}{t}
			\mathcal{O}_{\widetilde{Y}} + f^*
			\Omega^1_{\widetilde{Y}} = \\
			& \dfrac{ \dif z^m}{z^{\alpha m}}
			\mathcal{O}_{\widetilde{Y}} + f^*
			\Omega^1_{\widetilde{Y}} =
			\dfrac{ z^{m-1} \dif z}{z^{\alpha m}}
			\mathcal{O}_{\widetilde{Y}} + f^*
			\Omega^1_{\widetilde{Y}} = \dfrac{\dif z}{z}
			z^{m(1-\alpha)} \mathcal{O}_{\widetilde{Y}} + 
			f^*\Omega^1_{\widetilde{Y}} = \mathcal{E}_\alpha
		\end{split}
	\end{equation*}
\end{remark}
locally around $p$. The remaining cases are dealt with analogously.

\subsection{Chern classes of
\texorpdfstring{$\mathcal{E}_\alpha$}{Ea}}
We are going to express the Chern classes of $\mathcal{E}_\alpha$ in terms of 
log resolution data. Refer to \S \ref{subsection:notation} for the notation 
regarding resolution data. It is worth noting that these computations can be
performed directly by using \eqref{equation:Ealphaexact}, but in view of Remark
\ref{remark:Ealphaaspullback}, $\mathcal{E}_\alpha =\widehat{f^*\Omega}^1_{
\widetilde Y} \left(\log (\alpha \widetilde{C}+\widetilde{D}+E)\right) $ 
according to Langer's notation, see \cite[\S 2, 3]{langer-orbifold}, 
and then the
Chern numbers of $\mathcal{E}_\alpha$ are equal to the Chern numbers of the
pair
$(\widetilde{Y},\alpha \widetilde{C} + \widetilde{D} + E)$. In particular by
\cite[p.~368, Theorem 3.6]{langer-orbifold}, the computation of
$c_2(\mathcal{E}_\alpha)$ can be reduced to the computation of
$e_{\mathrm{orb}}(\widetilde{Y},\alpha \widetilde{C} + \widetilde{D} + E)$, 
the orbifold Euler number of the pair. In what follows we will pursue this
latter approach, tacitly adopting Langer's notation when necessary. 
Before proceeding further let us introduce:
\begin{notation}
For $i=1,\ldots, s+r$ let $m_i$, $\delta_i$ and $\epsilon_i$ non negative
integers such that 
\begin{equation}\label{equation:mi}
  \widetilde{C} = \pi^* \mu^* C - \sum_{i=1}^{s+r} m_i \overline{E}_i\ ,
\end{equation} 
\begin{equation} \label{equation:deltai}
  \widetilde{D} = \pi^* \mu^* D - \sum_{i=1}^{s+r} \delta_i
  \overline{E}_i\ , 
\end{equation}
\begin{equation}\label{equation:epsiloni}
  E = \sum_{i=1}^{s} \overline{E}_i - \sum_{j=s+1}^{s+r} \epsilon_j
  \overline{E}_j\ , 
\end{equation}
respectively. The above integers can be defined recursively, 
see the proof of Lemma
\ref{lemma:xed}, and $m_i\ge1$, $i=1,\ldots, s+r$, since at each step the blow
up occurs in a point of a strict transform of $C$.
\end{notation}
\begin{remark}
  Observe that 
  \begin{equation}\label{eqution:eprime}
    E^\prime = \sum_{i=1}^{s^\prime} \overline{E}_i -
    \sum_{j=s+1}^{s+r^\prime} \epsilon_j
    \overline{E}_j\ ,
  \end{equation}
  and 
  \begin{equation}\label{equation:esecond}
    E^{\prime \prime} = \sum_{i=s^\prime+1}^{s} \overline{E}_i - 
    \sum_{j=s+r^\prime+1}^{s+r} \epsilon_j
    \overline{E}_j\ ,
  \end{equation}
  moreover
  \begin{equation} \label{equation:deltaj}
    \delta_j=0\ \text{for $1\le j \le s^\prime$ and $ s+ 1\le j \le s+ 
    r^\prime $ .}
  \end{equation}
\end{remark}

\begin{lemma} \label{lemma:xed}
  There exist integers $x_i$, $i=1,\ldots,s+r$ such that: 
  \begin{equation}\label{equation:xi}
    K_{\widetilde{Y}}+\widetilde{D} + E= \pi^* \mu^* (K_X + D ) +
    \sum_{i=1}^{s+r}x_i \overline{E}_i 
  \end{equation}
  and  
  \begin{gather}
    x_i\ge 0\ \text{for $i=1,\ldots,s+r$}, \label{equation:xi2}\\  
    x_i=2\ \text{for $1\le i\le s^\prime$},\ x_i \le 1 
    \ \text{for $s^\prime + 1
    \le i \le s$ ,} \label{equation:xsmall} \\
    x_j+\delta_j+\epsilon_j = 1 \ \text{for $ s+1 \le j \le s+r$}
    \label{equation:xpdpe}
  \end{gather}
\end{lemma}

\begin{proof}
  For $k=0,\ldots ,r$, denote by $D^k$ the strict transform of
  $D$ in $Y_k$ and by $E_i^k$, $\overline{E}_i^k$, $i=1,\ldots, s+k$,
  the strict transform and
  the total transform in $Y_k$ of the exceptional curve $E_i$, $i=1,\dots k$
  respectively. Where in particular 
  $\overline{E}^k_{s+k}=E^k_{s+k}$ and $ \overline{E}^0_i = E^0_i$, $i=1,\dots s$.
  It follows that for $k=0$
  \begin{equation*}
    D^0 = \mu^* D - \sum_{i=1}^{s} \delta_i
    \overline{E}^0_i
  \end{equation*}
  and for $k\ge 1$
  \begin{equation*}
    D^k = \pi_k^* \ldots \pi_1^* \mu^* D - \sum_{i=1}^{s+k} \delta_i
    \overline{E}^k_i\ ,
  \end{equation*}
  moreover
  \begin{equation*}
    E^k = \sum_{i=1}^{s} \overline{E}^k_i - \sum_{j=s+1}^{s+k} \epsilon_j
    \overline{E}^k_j\ ,
  \end{equation*}
  where we put $E^k = \sum_{i=1}^{s+k} E_i^k$. By the formula
  for the canonical divisor of a blow-up there exist integers $x_1,\ldots,
  x_s$ such that:
  \begin{equation} \label{equation:}
    K_{Y_0} + D^0 + E^0 = \mu^* \left( K_X + D \right) +
    \sum_{i=1}^{s} x_i \overline{E}_i^{0}\ ,
  \end{equation}
  and 
  \begin{equation*}
    x_i = \begin{cases}
      2 & 1\le i \le s^\prime \\
      0 & \text{$s^\prime + 1 \le i \le s$, and $E_i^0$ maps onto a double 
	point of 
      $D$} \\
      1 & \text{$s^\prime + 1 \le i \le s$, and $E^0_i$ maps onto a 
      smooth point of $D$ } \ .
      \end{cases}
  \end{equation*}
  Arguing now by induction, put $\pi_0 \coloneqq id_Y$, suppose that for $k\ge
  0$ we have
  \begin{equation*}
    K_{Y_k} + E^k + D^k = \pi^*_k\dots \mu^* 
    \left( K_X + D \right) + \sum_{i=1}^{s+k} x_i \overline{E}_i^{k}
  \end{equation*}
  then pulling back both sides of the above equality by $\pi_{k+1}$ 
  we obtain
  \begin{equation*}
    K_{Y_{k+1}} + E^{k+1} + D^{k+1} + \left(
    \epsilon_{s+k+1}+ \delta_{s+k+1} -1 \right) E_{s+k+1}^{k+1}
    = \pi^*_{k+1}\dots \mu^* 
    \left( K_X + D \right) + \sum_{i=1}^{s+k} x_i \overline{E}_i^{k+1}   
  \end{equation*}
  and since $\overline{E}_{s+k+1}^{k+1}= E_{s+k+1}^{k+1}$ 
  \begin{equation*}
    K_{Y_{k+1}} + E^{k+1} + D^{k+1}
    = \pi^*_{k+1}\dots \mu^* 
    \left( K_X + D \right) + \sum_{i=1}^{s+k+1} 
    x_i \overline{E}_i^{k+1}   
  \end{equation*}
  with $x_{s+k+1} = \left( 1 - \epsilon_{s+k+1}- \delta_{s+k+1} \right) $.
  Observe that $ \epsilon_{s+k+1}+ \delta_{s+k+1}\le 1 $, indeed after
  the first $s$ blow-ups, the center of the blow-up can not be a point where
  a component of the strict transform of $D$ and two exceptional curves
  meet since $D$ is assumed to be simple normal crossing. This
  concludes the proof.
\end{proof}

\begin{notation}
	$g(C)$, $g(D_k)$ denote the  
	geometric genus of $C$, $D_k$, respectively.
\end{notation}

\begin{remark} \label{remark:explicitep}
In view of Definition \ref{notation:relativepoincare}, since
$D$ is simple normal crossing and the topological 
Euler-Poincar\'e characteristic is
additive under disjoint unions, we have 
\begin{equation}\label{equation:epopensurface}
	e_{X\setminus D}= e_{\mathrm{top}}(X)-\sum_i \left(2-2g(D_i)\right)+
  \sum_{i<j}D_i\cdot D_j\ ,
\end{equation}
in other words $e_{X\setminus D} = e_{\mathrm{orb}}(X\setminus D)$, see for instance
\cite[pag.~359]{langer-orbifold}.
Moreover considering the normalization $\eta \colon \widetilde{C}\rightarrow C$ 
obtained by restricting 
$\pi\circ \mu$ to $\widetilde{C}$, then
\begin{equation}\label{equation:epopencurve}
	e_{C\setminus D} = e_{\mathrm{top}}\left( \widetilde{C }\setminus
	\eta^{-1}(D)\right) =e_{\mathrm{top}}( \widetilde{C} )- 
  \sharp \left( \eta^{-1}(D)\right) 
  =2-2g(C)-(E^{\prime
  \prime}+\widetilde{D})\cdot \widetilde{C}\ ,
\end{equation}
namely $e_{X\setminus D}$, $e_{C\setminus D}$ are both orbifold Euler numbers. 
\end{remark}

\begin{lemma}\label{lemma:chernresoldata}
  \begin{gather*}
    c_2(\mathcal{E}_\alpha) = d \left( e_{X\setminus D} - \alpha 
    e_{C\setminus D} + \sum_{i=1}^{s^\prime}(\alpha m_i
    -1) - \sum_{j=s+1}^{s+r^\prime}\alpha m_j \epsilon_j \right) \\
    c_1( \mathcal{E}_\alpha) = f^* \left( \pi^* \mu^*
    (K_X+D + \alpha C) + \sum_{i=1}^{s+r} (x_i -\alpha m_i) 
    \overline{E}_i 
    \right)
  \end{gather*}
  and 
  \begin{equation*}
    c^2 _1(\mathcal{E}_\alpha)= d \left( (K_X+D)^2 + 2 \alpha ( K_X + D ) \cdot 
      C +
    \alpha^2 C^2 - \sum_{i=1}^{s+r}( x_i - \alpha m_i )^2 \right)\ ,
  \end{equation*}
  where $d$ denotes the degree of the finite map $f$.
\end{lemma}

\begin{proof}
	Let us start by computing the second Chern class. By Remark 
	\ref{remark:Ealphaaspullback}, $\mathcal{E}_\alpha =
	\widehat{f^*\Omega}^1_{\widetilde Y} \left(\log 
	(\alpha \widetilde{C}+\widetilde{D}+E)\right) $ and then by
	\cite[Theorem 3.6]{langer-orbifold} we have that 
	\begin{equation*}
		c_2(\mathcal{E}_\alpha) = d\ e_{\mathrm{orb}}(\widetilde{Y}, 
			\alpha
		\widetilde{C} + \widetilde{D} + E)\ .
	\end{equation*}
	Denote by $\mathcal{S}$ the set of singular points of
	$\mathrm{Supp}\ \alpha \widetilde{C} + \widetilde{D} + E$ and observe
	that in the present situation local orbifold Euler numbers are zero.
	Indeed, local orbifold Euler numbers depend only on the analytic type
	of the singularity \cite[p.~366, 3.2]{langer-orbifold}, and they are
	zero in the case of normal crossings \cite[Lemma 3.3]{langer-orbifold}.
	Since each blow-up increases the topological Euler number of the surface
	by one,
	it follows that \cite[Definition 3.1]{langer-orbifold} reads
	\begin{align*}
		e_{\mathrm{orb}}&(\widetilde{Y}, \alpha \widetilde{C} +
		\widetilde{D} + E ) = e_{\mathrm{top}}(X) + r + s +
		\sum_{k}(2-2g(D_k)) -2r -2s \\
		& + 2\ \sharp\ ( \mathcal{S} \setminus
		\widetilde{C} ) + \sharp\ (\mathcal{S}\cap \widetilde{C})
		-\alpha \left( 2 - 2g(C) - \sharp\ (\mathcal{S} \cap
		\widetilde{C}) \right) - \sharp\ \mathcal{S} =\\
		& e_{\mathrm{top}}(X) +
		\sum_{k}(2-2g(D_k)) -r -s 
		+  \sharp\ ( \mathcal{S} \setminus
		\widetilde{C} )
		-\alpha \left( 2 - 2g(C) - \sharp\ (\mathcal{S} \cap
		\widetilde{C}) \right)\ .
	\end{align*}
	The number of double points of $E$ is equal to $r$ by construction
	and then
	\begin{equation*}
	       \sharp\ (\mathcal{S} \setminus
	       \widetilde{C}) = r + E\cdot \widetilde{D} + 
	       \sum_{i<j}\widetilde{D}_i \cdot \widetilde{D}_j\ ,
        \end{equation*}
	as a consequence we have
	\begin{equation} \label{equation:chernea}
		\begin{split}
		c_2&(\mathcal{E}_\alpha)= 
		d \left( e_{\mathrm{top}}(X) - s - \sum_k (2- 2g(D_k)) +
		E \cdot \widetilde{D}+ \right. \\
		& \left. \sum_{i<j}\widetilde{D}_i
			\cdot \widetilde{D}_j +
			\alpha \left( 2 - 2g(C) \right)+ \alpha
			( E+\widetilde{D}
		)\cdot \widetilde{C} \right) \ .
	        \end{split}
	\end{equation}
  Observe that $ \sum_{i<j}D_i \cdot D_j $ equals the number of double points
  of $D$. Each blow-up in a double point of $D$ increases $E\cdot
  \widetilde{D}$ by two and a blow-up in a smooth point by one. It follows that
  \begin{equation*}
    \sum_{i<j}D_i \cdot D_j = E \cdot \widetilde{D}-(s-s^\prime)+
    \sum_{i<j}\widetilde{D}_i \cdot \widetilde{D}_j
  \end{equation*}
  and then in view of 
  \eqref{equation:epopensurface}, \eqref{equation:epopencurve} and
  \eqref{equation:chernea} we finally have
  \begin{equation*}
    c_2(\mathcal{E}_\alpha) = d\left( e_{X\setminus D} - \alpha
    e_{C\setminus D} + \alpha E^\prime \cdot 
    \widetilde{C} - s^\prime \right)\ . 
  \end{equation*}
  Since
  \begin{equation*}
    E^\prime \cdot \widetilde{C} = \left( \sum_{i=1}^{s^\prime}
    \overline{E}_i - \sum_{j=s+ 1}^{s+r^\prime}\epsilon_j
    \overline{E}_j \right) \cdot \left( \pi^* \mu^* C - \sum_{i=1}^{s+r} m_i
    \overline{E}_i \right ) = \sum_{i=1}^{s^\prime} m_i -
    \sum_{j=s+1}^{s+r^\prime} m_j \epsilon_j
  \end{equation*}
  we can conclude that 
  \begin{equation*}
    c_2(\mathcal{E}_\alpha) = d \left( e_{X\setminus D} - \alpha
    e_{C\setminus D} + \sum_{i=1}^{s^\prime}(\alpha m_i
    -1) - \sum_{j=s+1}^{s+r^\prime}\alpha m_j \epsilon_j \right)\ .
  \end{equation*}
  Regarding $c_1(\mathcal{E}_\alpha)$, note that $c_1(\mathcal{E}_\alpha) 
  = f^* \left( K_{\widetilde{Y}} + \widetilde{D} + E + \alpha \widetilde{C} 
  \right)$. For instance, this can be deduced by a direct computation from
  \eqref{equation:Ealphaexact}. 
  In view of
  \eqref{equation:mi} and \eqref{equation:xi} we
  have that 
  \begin{align*}
    c_1(\mathcal{E}_\alpha) =& f^* \left( K_{\widetilde{Y}} +
    \widetilde{D} + E + \alpha \widetilde{C} \right) \\
    =& f^* \left( \pi^* \mu^*
    (K_X+D) + \sum_{i=1}^{s+r} x_i \overline{E}_i + \alpha \pi^* \mu^* C -
    \sum_{i=1}^{s+r} \alpha m_i \overline{E}_i \right) \\
    =& f^* \left( \pi^* \mu^*
    (K_X+D + \alpha C) + \sum_{i=1}^{s+r} (x_i -\alpha m_i) \overline{E}_i 
    \right)
  \end{align*}
  and then 
  \begin{equation*}
    c^2 _1(\mathcal{E}_\alpha) = d \left( (K_X+D)^2 + 2 \alpha ( K_X + D )\cdot
      C +
    \alpha^2 C^2 - \sum_{i=1}^{s+r}( x_i - \alpha m_i )^2 \right)\ .
  \end{equation*}
\end{proof}

\subsection{Adjunction formula}
Following the notation introduced in the proof of Lemma \ref{lemma:xed}, 
let us denote by $D^0$, $C^0$ the
strict transforms by the blow-up $\mu$ 
of $D$ and $C$, respectively. We have
\begin{gather}
  \widetilde{D} = \pi^* D^0 - \sum_{j=s+1}^{s+r}\delta_j
  \overline{E}_j = \pi^* D^0 - \sum_{j=s+r^\prime+1}^{s+r}
  \delta_j
  \overline{E}_j\label{equation:d0}\ , \\
  \widetilde{C} = \pi^* C^0 - \sum_{j=s+1}^{s+r} m_j
  \overline{E}_j \label{equation:c0}\ ,
\end{gather}
where the second equality in \eqref{equation:d0} follows by
\eqref{equation:deltaj}.

\begin{proposition} \label{proposition:adjunction}
  \begin{equation*}
    (K_X+  D) \cdot C + C^2  =
     -e_{C\setminus D} 
    +\sum_{i=1}^{s} m_i(m_i-x_i+1) + 
      \sum_{j=s+1}^{s+r} m_j (m_j - 1) 
       + \sum_{j=s+r^\prime + 1}^{s+r} m_j (\epsilon_j + \delta_j)\ . 
  \end{equation*}
\end{proposition}

\begin{proof}
  By \eqref{equation:xi} and \eqref{equation:mi} we have
  \begin{equation}\label{equation:proofpropadj1}
    ( K_{\widetilde{Y}} + \widetilde{D} + E + \widetilde{C} ) \cdot
    \widetilde{C} = ( K_X + D ) \cdot C + C^2 + \sum_{i=1}^{s+r}m_i (
    x_i - m_i)\ ,
  \end{equation}
  observe that by adjunction:
  \begin{equation*}
   2 g(C) - 2 =
   ( K_{\widetilde{Y}}+ \widetilde{C} ) \cdot \widetilde{C}\ ,
  \end{equation*}
  moreover
  by \eqref{equation:mi} and \eqref{equation:epsiloni}: 
  \begin{equation*}
    E\cdot \widetilde{C} = \left(\sum_{i=1}^{s} \overline{E}_i - 
      \sum_{j=s+1}^{s+r} 
    \epsilon_j
  \overline{E}_j \right) \cdot \left(\pi^* \mu^* C - \sum_{i=1}^{s+r} m_i
  \overline{E}_i \right) =  \sum_{i=1}^{s}m_i -\sum_{j=s+1}^{s+r}m_j 
  \epsilon_j\ ,
  \end{equation*}
  and by \eqref{equation:d0} and \eqref{equation:c0}: 
  \begin{equation*}
    \widetilde{D}\cdot \widetilde{C} = \left(\pi^* D^0 - 
    \sum_{j=s+1}^{s+r}\delta_j
  \overline{E}_j  \right) \cdot \left( \pi^* C^0 - 
  \sum_{j=s+1}^{s+r} m_j
  \overline{E}_j \right) = D^0\cdot C^0 -
  \sum_{j=s+1}^{s+r}m_j \delta_j\ .
  \end{equation*}
  Substituting the above equalities and \eqref{equation:xpdpe}, 
  \eqref{equation:epopencurve}
  in
  \eqref{equation:proofpropadj1} we get:
  \begin{equation*}
    ( K_X+ D) \cdot  C + C^2 =
    - e_{C\setminus D} + 
    D^0 \cdot C^0 - (E^{\prime
    \prime}+\widetilde{D}) \cdot \widetilde{C} + \\
     \sum_{i=1}^{s}m_i(m_i-x_i+1) +
    \sum_{j=s+1}^{s+r} m_j (m_j - 1)\ ,
  \end{equation*}
  but by \eqref{equation:esecond}, \eqref{equation:d0} and
  \eqref{equation:c0}:
  \begin{equation*}
    \begin{split}
      (E^{\prime \prime} & + \widetilde{D}) \cdot \widetilde{C} =  \\
      & \left( \sum_{i=s^\prime+1}^{s} \overline{E}_i - 
	\sum_{j=s+r^\prime+1}^{s+r} \epsilon_j
	\overline{E}_j + \pi^* D^0 - \sum_{j=s+1+r^\prime}^{s+r}
	\delta_j
      \overline{E}_j \right) \cdot \left( \pi^* 
	C^0 - \sum_{j=s+1}^{s+r}\delta_j
      \overline{E}_j \right) = \\
      & D^0 \cdot C^0 -
      \sum_{j= s+ r^\prime+1 }^{s+r} m_j (\epsilon_j + \delta_j)\ ,
    \end{split}
  \end{equation*}
  and this concludes the proof of the Proposition.
\end{proof}

\section{Proof of Theorem \ref{theorem:main} }
Given $X$, $D$, $C$ and $\alpha$ as in the statement of
Theorem \ref{theorem:main}, let us apply the construction of Section
\ref{section:mainconstruction}, refer to  
\S \ref{subsection:notation} 
for the notation regarding resolution data. 
We obtain a smooth surface $\widetilde{Y}$, 
a finite morphism $f\colon Z\rightarrow \widetilde{Y}$,
an orbibundle $\mathcal{E}_{\alpha}$
defined on $Z$, etc. 
Consider the divisor
\begin{equation} \label{equation:dalpha}
  \begin{split}
    D_{\alpha} & \coloneqq K_{\widetilde{Y}}+\widetilde{D}+E+\alpha
    \widetilde{C} = \pi^* \mu^* (K_X + D ) +
    \sum_{i=1}^{s+r}x_i \overline{E}_i + \alpha \widetilde{C} \\
    & = 
    \pi^* \mu^*
    (K_X+D + \alpha C ) + \sum_{i=1}^{s+r} (x_i -\alpha m_i) 
    \overline{E}_i 
  \end{split}
\end{equation}
where the second equality follows by \eqref{equation:xi} and the third one by
\eqref{equation:mi}. Since by hypothesis $K_X+D$ is $\mathbb{Q}$-effective, 
it follows
by \eqref{equation:xi2} and the second equality in \eqref{equation:dalpha} that
$D_\alpha$ is $\mathbb{Q}$-effective, then in turn the integral divisor
$c_1(\mathcal{E}_\alpha)=f^*D_\alpha$ is
$\mathbb{Q}$-effective and admits Zariski decomposition. In order to prove
Theorem \ref{theorem:main}, we might apply Theorem \ref{theorem:bmy} to the
bundle $\mathcal{E}_\alpha$. Indeed by Lemma \ref{lemma:chernresoldata}
we have 
\begin{equation}\label{equation:proofattempt}
  \begin{split}
    \frac{1}{d} & \left( 3c_2(\mathcal{E}_\alpha)-
    c^2_1(\mathcal{E}_\alpha ) \right)= 
    3 e_{X\setminus D} - (K_X+D)^2 - 2 \alpha \left[ (K_X+D) \cdot C +
    \frac{3}{2}e_{C\setminus D} \right] - \alpha^2 C^2 \\
    & +\sum_{i=1}^{s^\prime}\left( 3(\alpha m_i-1)+(x_i-\alpha m_i)^2\right)+
    \sum_{i=s^\prime+1}^{s} (x_i-\alpha m_i)^2 
    +\sum_{j=s+1}^{s+r^\prime}\left( -3\alpha m_j \epsilon_j 
    + (x_i-\alpha m_i)^2  \right) \\
    & + \sum_{j=s+r^\prime + 1}^{s+r} (x_i-\alpha
    m_i)^2 
  \end{split}
\end{equation}
and we might find a suitable upper bound for the sums in the right hand of
the above
inequality and finally replace it with the aid of
Proposition \ref{proposition:adjunction}. 
But from the third equality in
\eqref{equation:dalpha}, $D_\alpha$ is  not in general nef  
on $E$ and analogous considerations hold for 
$c_1(\mathcal{E}_\alpha)$ and $\mathrm{Supp}\ f^{*}(E)$. 
Moreover, 
it is not clear how to, if possible at all, bound
the summands on the right hand side of \eqref{equation:proofattempt}
by the summands on the right hand side of the equality in Proposition
\ref{proposition:adjunction}. In order to avoid the first problem, 
we could add the term
$\frac{f^* \overline{N}_\alpha^2}{4}$ on the left hand side of 
\eqref{equation:proofattempt},  where $\overline{N}_\alpha$ 
stands for the negative part of
the absolute Zariski decomposition of $D_\alpha$, and consequently subtract
\footnote{Recall that $\overline{N}^2_\alpha<0$ .}
a corresponding
quantity to the summands on the right-hand side of the equality. 
But it turns out that the
denominator $4$ is too big to obtain the desired bound. In order to 
get rid of this denominator for as many terms as possible, following
Miyaoka,  
we first perform a nef reduction of
$\mathcal{E}_\alpha$ contracting a part of the negative locus 
and finally apply Theorem \ref{theorem:bmy} to the resulting 
bundle. 

Consider the Zariski decomposition of $D_\alpha = P_{G+F^{\prime
\prime}} \left( D_\alpha \right) +
N_{G+F^{\prime\prime}}\left( D_{\alpha} \right)$ into its 
$(G+F^{\prime \prime})$-nef and $(G+F^{\prime
\prime})$-negative part, respectively. 
In order to simplify notation set $P_\alpha =  P_{G+F^{\prime
\prime}} \left( D_\alpha \right) $ and $N_\alpha
=N_{G+F^{\prime\prime}}\left( D_{\alpha} \right) $.
By unicity of the Zariski
decomposition we have that $ f^* D_\alpha = f^* P_\alpha +
f^* N_{\alpha} $ gives the Zariski decomposition of $c_1(
\mathcal{E}_\alpha ) = f^* D_\alpha$ into its $(\mathrm{Supp}\ f^{*}(G+F^{\prime
\prime}))$-nef and $(\mathrm{Supp}\ f^{*}(G+F^{\prime \prime}))$-negative 
part, respectively.  
We are going to perform a nef reduction of $\mathcal{E}_\alpha$ 
according to Proposition \ref{proposition:nefreduction}, but
observe that the effective $\mathbb{Q}$-divisor 
$f^*N_\alpha$ with simple normal crossing support in not a priori
necessarily integral.
Consider then $\hat{f}$, the composition of $f$ 
and a Kawamata covering, \cite[4.1.12]{lazarsfeld-positivityI}, 
such that $\hat{f}^*N_\alpha$ is an integral divisor,
$\mathrm{Supp}\ \hat{f}^{*}(\overline{C+D})$ 
has again simple normal crossings\footnote{Observe 
  indeed that the support of $N_\alpha$ is 
  contained in the support of
$\overline{C+D}$ .} and
$\hat{f}^* \mathcal{E}_\alpha\subseteq 
\hat{f}^*\Omega_Z(\log \overline{C+D}) \subseteq \Omega_{Z}(\log(
\mathrm{Supp}\ \hat{f}^{*}(\overline{C+D})))$. With a slight abuse
of notation we will still
denote $\hat{f}$ by $f$ and $\hat{f}^*\mathcal{E}_\alpha$ by
$\mathcal{E}_\alpha$ since this does not generate any confusion.
Indeed observe that all the results in Section
\ref{section:mainconstruction} are still valid with
$\hat{f}^*\mathcal{E}_\alpha$ in place of $\mathcal{E}_\alpha $ and
$\hat{f}$ in place of $f$.
Let us apply Proposition \ref{proposition:nefreduction} to the map 
$\rho \coloneqq \mu_2 \circ \pi \circ f$, $\rho\colon Z \rightarrow Y^\prime$, 
and the rank two vector bundle 
$$\mathcal{E}_\alpha \subseteq
f^*\Omega_{\widetilde{Y}} (\log 
\overline{C+D} ) \subseteq \Omega_{Z}(\log 
(\mathrm{Supp}\ f^{*}(\overline{C+D})))\ .$$ 
With the notation introduced in the Proposition, if we set 
$$\Gamma = E_1 +
\cdots E_{s^\prime}$$ 
and 
$$\Lambda = \Gamma+\mathrm{Supp}\ \mu_2^{*}(D)$$
the exceptional
locus of $\rho$ is the simple normal crossing divisor 
$$R =
\mathrm{Supp}\ f^{*}(F^{''}+G)\ ,$$
that in turn is contained in the simple normal crossing
divisor 
$$\Delta =\mathrm{Supp}\ f^{*}(\overline{C+D})$$ 
and moreover $\rho ( R )\subset
\Lambda $. Note that
\begin{equation*}
  \rho^{*} \Omega_{Y^\prime}( \log\
  \Lambda ) \subseteq 
  \Omega_{Z}( \log(\mathrm{Supp}\ \rho^{*} 
  ( \Lambda ) ) \subseteq \mathcal{E}_\alpha\
  ,
\end{equation*}
since $\mathrm{Supp}\ \rho^{*} ( \Lambda ) =\mathrm{Supp}\ ( \mu_2 \circ \pi
\circ f )^{*} 
(\Gamma+\mathrm{Supp}\ \mu_2^{*}(D) )
$ does not contain $\mathrm{Supp}\ f^{*}\widetilde{C}$. 
We obtain a rank two vector bundle $\widetilde{\mathcal{E}}_\alpha
=\mathfrak{P}_{\mathrm{Supp}\ f^{*}(F^{''}+G)}( \mathcal{E}_\alpha )$ 
such that $c_1( \widetilde{\mathcal{E}}_\alpha ) = f^* P_\alpha$ and
$c_2(\widetilde{\mathcal{E}}_\alpha ) = c_2( \mathcal{E}_\alpha )$. 
Since $N_\alpha$ is supported on $G + F^{''}$, we set 
\begin{equation}\label{equation:negativepart}
  N_\alpha = \sum_{j=s^{'}+1}^{s+r} b_j \overline{E}_j\ ,
\end{equation}
and we are in position to formulate the following:

\begin{proposition} \label{proposition:firstest}
  \begin{align*}
    \frac{1}{d}& \left( 3c_2(\widetilde{\mathcal{E}}_\alpha) - 
    c^2_1(\widetilde{\mathcal{E}}_\alpha) \right)=
    \frac{1}{d} \left( 3c_2(\mathcal{E}_\alpha)-
    c^2_1(\mathcal{E}_\alpha) + f^*N^2_\alpha \right)= \\
    & 3 e_{X\setminus D} - (K_X+D)^2 - 2 \alpha \left[ (K_X+D) \cdot C +
    \frac{3}{2}e_{C\setminus D} \right] - \alpha^2 C^2 \\
    & +\sum_{i=1}^{s^\prime}\left( 3(\alpha m_i-1)+(x_i-\alpha m_i)^2\right)+
    \sum_{i=s^\prime+1}^{s}\left( (x_i-\alpha m_i)^2 -b_i^2 \right) \\
    & +\sum_{j=s+1}^{s+r^\prime}\left( -3\alpha m_j \epsilon_j + (x_i-\alpha
    m_i)^2 -b_j^2 \right) + \sum_{j=s+r^\prime + 1}^{s+r}\left( (x_i-\alpha
    m_i)^2 -b_j^2 \right)
  \end{align*}
\end{proposition}

\begin{proof}
  In view of \eqref{equation:negativepart} and Lemma
  \ref{lemma:chernresoldata}, the Proposition follows by a
  straightforward computation.
\end{proof}

We are going to find an upper bound for each of the last three summations 
contained in the right
hand of the equality in Proposition \ref{proposition:firstest}. 
First of all we have:

\begin{lemma} \label{lemma:bj}
  \begin{equation*}
    b_j \ge x_j - \alpha m_j \quad \text{for $ s^\prime + 1 \le j \le s+r$}
    \ ,
  \end{equation*}
  and then
  \begin{equation} \label{equation:bj}
    b_j^2\ge \max\{x_j-\alpha m_j, 0\}^2 \quad \text{for $ s^\prime + 1 \le j 
    \le s+r$}.
  \end{equation}
\end{lemma}

\begin{proof}
  Let us prove the first inequality. If $j\ge s^\prime +1 $ then
  the upport of the effective divisor $\overline{E}_j$ 
  is contained in the support of
  $F^{\prime \prime}+ G$. It follows that:
  \begin{equation*}
    0 \le P_\alpha \cdot \overline{E}_j = \left( D_\alpha - N_\alpha 
    \right)\cdot \overline{E}_j =  -x_j+\alpha m_j + b_j\ ,
  \end{equation*}
  note indeed that the divisors $\overline{E}_j$, $j=1,\ldots,s+r$, 
  are
  orthogonal with respect to the intersection product.
  Regarding \eqref{equation:bj}, it is an immediate
  consequence of the first inequality given that $b^2_j\ge 0$. 
\end{proof}

From Lemma \ref{lemma:bj} we deduce the following inequalities:

\begin{corollary} \label{corollary:fst}
  \begin{gather}
    (x_i - \alpha m_i)^2 - b_i^2 \le \alpha^2 m_i(m_i-x_i)\quad \text{for
    $s^\prime + 1 \le i \le s$,}  \label{equation:firstsum} \\
    -3 \alpha m_j \epsilon_j + (x_j-\alpha m_j)^2 - b_j^2 \le \alpha^2 m_j
    (m_j - 1) \quad \text{for $s+1\le j \le s+r^\prime$,}
    \label{equation:secondsum} \\
    (x_i - \alpha m_i)^2 - b_i^2 \le \alpha^2 m_j (m_j-1+\epsilon_j +
    \delta_j)\quad \text{ for $s+r^\prime + 1 \le j \le s+ r$.}
    \label{equation:thirdsum} 
  \end{gather}
\end{corollary}

\begin{proof}
  Let us start proving \eqref{equation:firstsum}.
  If $x_i - \alpha m_i\le 0$ then $ x_i \le \alpha m_i$ and by
  \eqref{equation:bj} we have
  \begin{gather*}
    (x_i - \alpha m_i)^2 - b_i^2 \le x_i^2 -2\alpha m_ix_i + \alpha^2
    m_i^2\le\\
    \alpha m_i x_i -2 \alpha m_i x_i + \alpha^2 m_i^2 = \alpha^2
    m_i (m_i - x_i)\ .
  \end{gather*}
  If $ x_i - \alpha m_i > 0 $ then by \eqref{equation:bj}, $ -b^2_i \le
  - (x_i - \alpha m_i)^2 $ and \eqref{equation:firstsum} becomes $0\le
  \alpha^2 m_i(m_i-x_i)$, which is true since $x_i\le 1$ for $s^\prime + 1 \le
  i \le s$ by \eqref{equation:xsmall} of Lemma \ref{lemma:xed}.
  Let us consider \eqref{equation:secondsum}. 
  First of all, observe that, for $s+1 \le j \le s+r^\prime$, by 
  \eqref{equation:xpdpe} of Lemma \ref{lemma:xed}
  we have that $x_j+\delta_j+\epsilon_j = 1$, moreover $\delta_j=0$, see
  \eqref{equation:deltaj}, since 
  blow-ups occur in points that do not map onto  
  $D$, it follows that $x_j=1-\epsilon_j$. 
  If $x_i-\alpha m_i \le 0$, again by Lemma \ref{lemma:bj} we have that:
  \begin{gather*}
    -3\alpha m_j \epsilon_j + (x_j-\alpha m_j)^2 - b_j^2 \le -3\alpha m_j
    \epsilon_j + (x_j-\alpha m_j)^2 = \\
    -3\alpha m_j \epsilon_j + (1 -
      \epsilon_j-\alpha
    m_j)^2 =  -3\epsilon_j\alpha m_j  + (1-\epsilon_j)^2 -2
    (1-\epsilon_j)\alpha m_j + \alpha^2 m_j^2 = \\
    \alpha^2 m_j^2 - (2+\epsilon_j)\alpha m_j + (1-\epsilon_j)^2\le 
    \alpha^2 m_j^2 - (2+\epsilon_j)\alpha m_j + (1-\epsilon_j)\alpha m_j =
    \\
    \alpha^2 m_j^2 - (1+2\epsilon_j) \alpha m_j \le \alpha^2 m_j^2 -
    \alpha^2 m_j= \alpha^2 m_j(m_j-1)\ .
  \end{gather*}
  To conclude the proof of \eqref{equation:secondsum}, observe that if
  $x_j-\alpha m_j> 0$ then \eqref{equation:secondsum} reduces to
  $$-3\epsilon_j \alpha m_j \le \alpha^2 m_j(m_j-1) $$
  that is always true
  since the
  left hand of the inequality is always non positive and the right hand is
  always non negative.
  Finally let us prove \eqref{equation:thirdsum} beginning again with
  the case $x_j-\alpha m_j \le 0$. By \eqref{equation:bj} of
  Lemma \ref{lemma:bj} and Equation \eqref{equation:xpdpe} of Lemma
  \ref{lemma:xed} we have:
  \begin{gather*}
    (x_j - \alpha m_j)^2 -b_j^2 = (x_j - \alpha m_j)^2 = x_j^2 -2 x_j\alpha
    m_j + \alpha^2 m_j^2\le \\
    2\alpha^2 m_j^2 -2 x_j \alpha m_j \le \alpha^2 m_j^2 -
    x_j \alpha m_j \le \alpha^2 m_j^2 - \alpha^2 m_j x_j = \\
    \alpha^2 m_j(m_j - x_j) = \alpha^2 m_j ( m_j - 1 + \epsilon_j + \delta_j
    ) = \alpha^2 m_j(m_j - 1 )+ \alpha^2 m_j (\epsilon_j + \delta_j)\ .
  \end{gather*}
  In case $x_j - \alpha m_j > 0$, as before,
  \eqref{equation:thirdsum} reduces to $$0\le \alpha^2 m_j(m_j - 1 )+
  \alpha^2 m_j (\epsilon_j + \delta_j)\ ,$$ that is trivially true. 
\end{proof}

Observe that $P_\alpha$ 
may not be nef on $F+G$. Put $\widehat{P}_\alpha = P_{F+G}( P_\alpha ) $,
$\widehat{N}_\alpha = N_{F+G}( P_\alpha ) $ and 
$\overline{P}_\alpha = P( P_\alpha )$, 
$\overline{N}_\alpha = N( P_\alpha )$. 
$P_\alpha= \widehat{P}_\alpha + \widehat{N}_\alpha$ is then the Zariski
decomposition with support in $F+G$ and $P_\alpha = \overline{P}_\alpha +
\overline{N}_\alpha$ the absolute Zariski decomposition. Let us write
$\widehat{N}_\alpha$ as a sum\footnote{Since by construction $P_\alpha$ is nef on
  $F^{''}+ G$, we could write $\widehat{N}_\alpha =
  \sum_{i=1}^{s^\prime} \hat{b}_i \overline{E}_i + \sum_{j=1}^{s}
  \hat{b}_{s^\prime+j} \overline{E}_{s^\prime+j} $, indeed the remaining $\hat{b}_i$ are certainly
  zero. Since we will need to bound $\hat{b}_i$,
  $i=1,\ldots,s^\prime$,
  this would have no impact in what follows.}:
\begin{equation*}
  \widehat{N}_\alpha = \sum_{i=1}^{s+r} \hat{b}_i \overline{E}_i\ . 
\end{equation*}

\begin{lemma} \label{lemma:}
  \begin{equation*}
    \hat{b}_i \ge x_i - \alpha m_i \quad \text{for $ 1 \le i 
    \le s^\prime$}
    \ ,
  \end{equation*}
  then
  \begin{equation*}
    \hat{b}_i^2\ge (\max\{x_i-\alpha m_i, 0\})^2 \quad \text{for $ 
      1 \le i 
    \le s^\prime $}
  \end{equation*}
  and
  \begin{equation} \label{equation:absolutezariski}
    \overline{N}_\alpha^2 \le \widehat{N}_\alpha^2 \le - \sum_{i=1}^{s^\prime}
    (\max\{x_i-\alpha m_i, 0\})^2\ .
  \end{equation}
\end{lemma}

\begin{proof}
  We follow the same argument as in the proof of Lemma \ref{lemma:bj}. 
  We have that
  \begin{equation*}
    0 \le \widehat{P}_\alpha \cdot \overline{E}_i = (P_\alpha -
    \widehat{P}_\alpha ) \cdot \overline{E}_i = 
    - x_i +\alpha m_i + \hat{b}_\alpha\ , 
  \end{equation*}
  for $i=1,\ldots,s+r$, indeed $\widehat{P}_\alpha$ is nef on $F+G$ and the
  support of the effective divisor 
  $\overline{E}_i$, for $i=1,\ldots,s+r$, is contained in
  $F+G$.
  It follows that $\hat{b}_i \ge x_i - \alpha m_i $ and
  moreover, since
  $\hat{b}^2_i$ is non negative, $ \hat{b}_i^2\ge (\max\{x_i-\alpha
  m_i, 0\})^2 $.
  Finally, observe that by Proposition
  \ref{proposition:zariskidecomposition} we have that $\overline{N}_\alpha^2 \le
  \widehat{N}_\alpha^2$ and then 
  \begin{equation*}
    \overline{N}_\alpha^2 \le \widehat{N}_\alpha^2 = - \sum_{i=1}^{s+r}
    \hat{b}_i^2 \le  - \sum_{i=1}^{s^\prime} \hat{b}_i^2 \le 
    - \sum_{i=1}^{s^\prime} (\max\{x_i-\alpha
    m_i, 0\})^2\ ,
  \end{equation*}
  that proves \eqref{equation:absolutezariski} and conclude the proof 
  of the Lemma. 
\end{proof}

We can summarize the above computations in the following:

\begin{proposition} \label{proposition:principalinequality}
  \begin{equation}\label{equation:principalinequality}
    \begin{split}
      \frac{1}{d} & \left(3c_2(\widetilde{\mathcal{E}}_\alpha) - 
	c^2_1(\widetilde{\mathcal{E}}_\alpha) +
      \frac{(f^*\overline{N}_\alpha)^2}{4} \right) \le
      3 e_{X\setminus D} - (K_X+D)^2 \\ 
      &-2 \alpha \left[ (K_X+D) \cdot C +
      \frac{3}{2} e_{ C\setminus D} \right] + 
      \frac{\alpha^2}{2} \left[ C^2 + 3(K_X+D) \cdot C + 3 e_{ C \setminus D } 
      \right]\ .
    \end{split}
  \end{equation}
\end{proposition}

\begin{proof}
  Combining Proposition \ref{proposition:firstest}, Corollary
  \ref{corollary:fst} and \eqref{equation:xsmall} we get:
  \begin{equation} \label{equation:finalstepone}
    \begin{split}
      \frac{1}{d} & \left(3c_2(\widetilde{\mathcal{E}}_\alpha) - 
	c^2_1(\widetilde{\mathcal{E}}_\alpha) +
      \frac{(f^*\overline{N}_\alpha)^2}{4} \right) \le \\
      & 3e_{X\setminus D} - (K_X+D)^2 - 2 \alpha \left[ (K_X+D) \cdot C +
      \frac{3}{2} e_{C\setminus D} \right] - 
      \alpha^2 C^2 \\
      & +\sum_{i=1}^{s^\prime}\left( 3(\alpha m_i-1)+(2 -\alpha m_i)^2 - 
	\frac{1}{4}(\max\{2 -\alpha
      m_i, 0\})^2 \right)+ \\
      & \sum_{i=s^\prime+1}^{s} \alpha^2 m_i(m_i-x_i)
      +\sum_{j=s+1}^{s+r^\prime} \alpha^2 m_j (m_j - 1) + 
      \sum_{j=s+r^\prime + 1}^{s+r} \alpha^2 m_j (m_j-1+\epsilon_j +
      \delta_j )\ .
    \end{split}
  \end{equation}

  We are going to bound the terms $3(\alpha m_i-1)+(2 -\alpha m_i)^2 -
  \frac{1}{4}(\max\{2 - \alpha m_i, 0\})^2 $ in the above sum. First of all
  observe that
  \begin{equation*}
    3(\alpha m_i-1)+(2 -\alpha m_i)^2 = 1-\alpha m_i+\alpha^2 m_i^2\, 
  \end{equation*}
  following Miyaoka, we claim that 
  \begin{equation} \label{equation:boundoldpoints}
    \begin{split}
      4 & \left[ 3(\alpha m_i-1)+(2 -\alpha m_i)^2 \right]-
      (\max\{2 - \alpha m_i, 0\})^2 = \\
      & 4 ( 1-\alpha m_i+\alpha^2 m_i^2 ) - (\max\{2 - \alpha m_i, 0\})^2 \le
      6\alpha^2 m_i(m_i-1)\ ,
    \end{split}
  \end{equation}
  for $i=1,\ldots,s^\prime$. If $2>\alpha m_i$, \eqref{equation:boundoldpoints}
  reduces to $3\alpha^2\le 6\alpha^2m_i(m_i-1)$ that it is true thanks to
  the fact that $m_i\ge 2$ for $i=1,\ldots,s^\prime$. Consider now the case
  $2\le \alpha m_i$, then
  \begin{equation*}
    \begin{split}
      4 & ( 1-\alpha m_i+\alpha^2 m_i^2 ) - 6\alpha^2 m_i(m_i-1)\le 
      2\alpha m_i - 4\alpha m_i + 4 \alpha^2 m_i^2 -6 \alpha^2 m_i(m_i-1)=
      \\
      & -2 \alpha m_i + 4 \alpha^2 m_i^2 - 6 \alpha^2 m_i (m_i-1)= 
      2\alpha m_i ( -1 + 2\alpha m_i -3\alpha m_i + 3\alpha ) = \\
      & 2\alpha m_i ( 3 \alpha -1 -\alpha m_i ) \le 2\alpha m_i(3-1-\alpha
      m_i)=  
      2 \alpha m_i (2 -\alpha m_i )\le 0\ ,
    \end{split}
  \end{equation*}
  and this completes the proof of \eqref{equation:boundoldpoints}. In view
  of \eqref{equation:boundoldpoints}, the inequality 
  \eqref{equation:finalstepone} becomes 
  \begin{equation*}
    \begin{split}
      \frac{1}{d} & \left(3c_2(\widetilde{\mathcal{E}}_\alpha) - 
	c^2_1(\widetilde{\mathcal{E}}_\alpha) +
      \frac{(f^*\overline{N}_\alpha)^2}{4} \right) \le \\
      & 3e_{X\setminus D} - (K_X+D)^2 - 2 \alpha \left[ (K_X+D) \cdot C +
      \frac{3}{2}e_{C\setminus D} \right] - \alpha^2 C^2 + \\
      & +\sum_{i=1}^{s^\prime} \frac{3}{2}\alpha^2 m_i(m_i-1)
      +\sum_{i=s^\prime+1}^{s} \alpha^2 m_i(m_i-x_i) \\
      &+\sum_{j=s+1}^{s+r} \alpha^2 m_j (m_j - 1) + 
      \sum_{j=s+r^\prime + 1}^{s+r} \alpha^2 m_j (\epsilon_j +
      \delta_j )\le  \\
      & 3 e_{X\setminus D} - (K_X+D)^2 - 2 \alpha \left[ (K_X+D) \cdot C +
      \frac{3}{2} e_{ C\setminus D} \right] - \alpha^2 C^2 + \\
      & + \frac{3}{2}\alpha^2 \left[ \sum_{i=1}^{s} m_i(m_i-x_i+1) + 
	\sum_{j=s+1}^{s+r} m_j (m_j - 1) 
	+ \sum_{j=s+r^\prime + 1}^{s+r} m_j (\epsilon_j + \delta_j) 
      \right]\ ,
    \end{split}
  \end{equation*}
  combining the above inequality and Proposition
  \ref{proposition:adjunction} we obtain
  \eqref{equation:principalinequality} and this concludes the proof of the
  Proposition.
\end{proof}

We are now ready to prove Theorem \ref{theorem:main}. 

\begin{proof}[Proof of Theorem \ref{theorem:main}]
  With the notation introduced above, 
  recall that by construction we have that $\mathrm{Supp}\ f^{*}(\overline{C+D})$ is a
  simple normal crossing effective divisor, 
  \begin{equation*}
    \widetilde{\mathcal{E}}_\alpha\subseteq \mathcal{E}_\alpha \subseteq
    f^*\Omega_Z \left(\log 
    \overline{C+D} \right) \subseteq \Omega_{Z} \left(\log
    (\mathrm{Supp}\ f^{*}(\overline{C+D})) \right)\ ,
  \end{equation*}
  moreover $c_1( \widetilde{\mathcal{E}}_\alpha) = P_\alpha $ is
  $\mathbb{Q}$-effective. 
  We can then apply Bogomolov-Miyaoka-Yau inequality, in the form of
  Theorem
  \ref{theorem:bmy}, to
  $\widetilde{\mathcal{E}}_\alpha$ obtaining 
  \begin{equation*}
    \left( 3c_2(\widetilde{\mathcal{E}}_\alpha) - 
      c^2_1(\widetilde{\mathcal{E}}_\alpha) +
    \frac{(f^*\overline{N}_\alpha)^2}{4}\right) \ge 0\ .
  \end{equation*}
  In view of the above inequality Theorem \ref{item:maininequality} follows
  directly by Proposition \ref{proposition:principalinequality}.

  Let us prove Theorem \ref{item:mainsecondinequality}. 
  Consider the inequality
  \eqref{equation:mainsecondary}, its left-hand side is 
  a quadratic polynomial in $\alpha$, 
  since $C$ is not a smooth $D$-rational curve
  and $(K_X + D)\cdot C \ge - \frac{3}{2} e_{X\setminus D}$ the leading
  coefficient is non positive. Indeed since $C$ is not a smooth $D$-rational
  curve we have that 
  \begin{equation*}
    C^2 + ( K_X+D ) \cdot C = 2g( C ) - 2 + \Sigma + D\cdot C \ge 0
  \end{equation*}
  where $\Sigma\ge 0$ is the contribution of the singular points of $C$ and
  $g\left( C\right)$ denotes its geometric genus as usual. 
  In view of $(K_X + D)\cdot C \ge - \frac{3}{2} e_{X\setminus D}$ 
  we have then
  \begin{equation} \label{equation:positiveleadingco}
    C^2 +  3 ( K_X + C ) \cdot C + 3 
    e_{C\setminus D} > 2  
    ( K_X + D )\cdot C + 3 e_{C \setminus D} 
    = 2 
    ( K_X + D )\cdot C + \dfrac{3}{2} e_{C\setminus D} > 0 \ .
  \end{equation}
  It follows that the left-hand side of the inequality
  \eqref{equation:maininequality} attains its minimum in correspondence of 
  \begin{equation*}
    \alpha_0 = \dfrac{2 \left( ( K_X + D )\cdot C +
    \dfrac{3}{2} e_{ C \setminus D} \right) }{C^2 +  3 
      ( K_X + C )\cdot C + 3 
    e_{ C\setminus D} }
  \end{equation*}
  and since $0\le \alpha_0 \le 1$, see \eqref{equation:positiveleadingco},
  substituting $\alpha_0$ in \eqref{equation:maininequality} we get 
  \begin{equation*}
    - \dfrac{2 \left( ( K_X + D ) \cdot C +
    \dfrac{3}{2} e_{C \setminus D} \right)^2 }{C^2 +  3 
    ( K_X + C ) \cdot C + 3 e_{ C\setminus D }} + 3
    e_{X \setminus D} - \left( K_X + D \right)^2\ge 0
  \end{equation*}
  from which \eqref{equation:mainsecondary} and then Theorem
  \ref{item:mainsecondinequality} follows
  immediately.
\end{proof}

\section{Proof of Theorem \ref{theorem:bound} and Corollary 
\ref{corollary:p1example}} \label{section:final}

\begin{notation}\label{notation:xsigmagamma}
  Suppose that $K_X + D$ is a nef and big divisor or equivalently that
  $\left( K_X + D \right)$ is nef and $( K_X + D )^2 > 0$. We
  introduce the following notation:
  \begin{gather*}
    x = \dfrac{( K_X + D ) \cdot C }{( K_X + D )^2} \\
    \sigma = \dfrac{ e_{X\setminus D} }{( K_X + D
    )^2 } \\
    \gamma = \dfrac{ -\frac{1}{2} e_{C\setminus D} }{\left( K_X + D
    \right)^2} \\
    y^2 = - \dfrac{\left( C - x ( K_X + D )
    \right)^2}{( K_X + D )^2 } = - \dfrac{C^2}{( K_X + D
    )^2} + x^2 \ .
  \end{gather*}
\end{notation}

\begin{remark} \label{remark:xsigmagamma}
  \begin{enumerate}
    \item $x\ge 0$, since we suppose that $ K_X + D $ is nef. 
    \item $\sigma \ge \tfrac{1}{3}$. Indeed, since we suppose $K_X +D$ 
      big, it is then linearly equivalent to an effective $
      \mathbb{Q}$-divisor and the inequality follows applying
      for instance \cite[Corollary 1.2]{miyaoka-maximal}.
    \item $y^2\ge 0$. Since $(K_X + D)^2>0$ the inequality
      follows applying Hodge Index Theorem to $ K_X + D
       $ and $E= \left[( K_X + D) \cdot C \right] ( K_X
      + D) - ( K_X + D )^2 C $. 
    \item $\gamma \ge - 1$, by definition of $e_{C\setminus D}$.
  \end{enumerate}
\end{remark}

The following Proposition is a consequence of Theorem \ref{theorem:main}.

\begin{proposition} \label{proposition:xsigmagamma}
  Let $X$ be a smooth projective surface, $D$ a simple normal
  crossing divisor on $X$ and $C$ an irreducible curve on $X$. Suppose that $K_X+D$
  is a nef and big divisor, $C$ is not a smooth $D$-rational curve
  and moreover that
  $(K_X + D)^2 > e_{X\setminus D}$. With Notation
  \ref{notation:xsigmagamma};
  \begin{enumerate}
    \item \label{item:polinomialP} If $( K_X + C )\cdot C> - 
      \dfrac{3}{2} e_{C \setminus D}$ or equivalently $x>3\gamma$, 
      we have that:
      \begin{equation}
	\label{equation:polinomialP}
	\mathcal{P}(x) = (\sigma -1) x^2 + (4\gamma+3\sigma -1) x - 2\gamma
	(3 \gamma + 3 \sigma -1) \ge 0
      \end{equation}
      and $x\le R_+(\sigma, \gamma)$, where $R_+(\sigma, \gamma)$ denotes the
      largest root of $\mathcal{P}$:
      \begin{equation}
	\label{equation:groot}
	R_{+}(\gamma, \sigma) = \dfrac{4\gamma + (3 \sigma -1) +
	  \sqrt{8(3\sigma -1)\gamma^2+ 8 \sigma (3 \sigma-1)\gamma +
	  (3\sigma -1)^2}}{2 (1-\sigma)}\ .
      \end{equation}
    \item \label{item:xbound} In general
      \begin{equation}
	\label{equation:xbound}
	x\le \max \{ 3\gamma, R_+(\sigma, \gamma)\} = R_+(\sigma, \gamma)\ .
	\end{equation}
  \end{enumerate}
\end{proposition}

\begin{proof}
  Let us start proving statement \eqref{item:polinomialP}. The
  hypotheses of Theorem \ref{theorem:main} are satisfied, consider then
  \eqref{equation:mainsecondary}, dividing by $\left[ \left( K_X + D \right)^2
  \right]^2>0$
  and substituting Notation \ref{notation:xsigmagamma} we obtain: 
  \begin{equation*}
     (\sigma -1) x^2 + (4\gamma+3\sigma -1) x - 2\gamma
     (3 \gamma + 3 \sigma -1) \ge \left( \sigma - \dfrac{1}{3} \right) y^2\
     ,
  \end{equation*}
  the left-hand side of the above inequality is $ \mathcal{P}(x)$ and the
  right-hand side is greater than or equal to zero by Remark
  \ref{remark:xsigmagamma}. Observe that since the curve $C$ 
  is not a smooth $D$-rational curve
  $\gamma \ge
  0$ and then by Remark \ref{remark:xsigmagamma} the discriminant of the
  polynomial $ \mathcal{P}$ is greater than or equal to zero:
  \begin{equation*}
    \Delta =8(3\sigma -1)\gamma^2+ 8 \sigma (3 \sigma-1)\gamma +
	  (3\sigma -1)^2 \ge 0\ ,
  \end{equation*}
  moreover since $\sigma < 1$ the leading coefficient of $
  \mathbb{P}$ is negative, it follows that $ x \le R_+(\gamma, \sigma)$,
  and \eqref{equation:groot} follows by a straightforward computation. 
  Statement \eqref{item:xbound} follows now immediately from part
  \eqref{item:polinomialP}. Indeed, since $\sigma \ge 1/3$ then
  $\tfrac{1}{1-\sigma}\ge \tfrac{3}{2}$ and 
  \begin{equation*}
    R_+(\sigma, \gamma) \ge \dfrac{4 \gamma}{2(1-\sigma)} =
    \dfrac{2 \gamma}{1 - \sigma} \ge 3 \gamma\ . 
  \end{equation*}
\end{proof}

We are now ready to prove Theorem \ref{theorem:bound}.
\begin{proof}[Proof of Theorem \ref{theorem:bound}]
  Let us start by proving Theorem \ref{item:uppercanonicalbound}. 
  With Notation \ref{notation:xsigmagamma}, by Proposition
  \ref{proposition:xsigmagamma} and taking into account that
  $\tfrac{3\sigma -1}{8 \gamma^2} - \tfrac{\sigma^2}{\gamma^2} =
  \tfrac{3\sigma -1 - 8 \sigma^2}{8 \gamma^2}\le 0$ we have that:
  \begin{gather*}
    x \le  R_+ (\sigma, \gamma) = \\
    \dfrac{4\gamma + (3 \sigma -1) + 2 \gamma \sqrt{2 (3 \sigma -1)} 
    \sqrt{ 1 + \dfrac{\sigma}{\gamma} + 
    \dfrac{3 \sigma -1 }{8 \gamma^2} } }{2 (1-\sigma)} \le \\
    \dfrac{4\gamma + (3 \sigma -1) + 2 \gamma \sqrt{2 (3 \sigma -1)} 
    \sqrt{ \left( 1 + \dfrac{\sigma}{\gamma}\right)^2 + 
    \dfrac{3 \sigma -1 }{8 \gamma^2} - 
    \dfrac{\sigma^2}{\gamma^2} } }{2 (1-\sigma)} \le  \\
    \dfrac{4\gamma + (3 \sigma -1) + 2 \gamma \sqrt{2 (3 \sigma -1)} 
    \sqrt{ \left( 1 + \dfrac{\sigma}{\gamma}\right)^2 } }
    {2 (1 -\sigma)} =
    \\
    \dfrac{2 + \sqrt{2(3\sigma -1)}}{1-\sigma}\gamma + \dfrac{(3 \sigma
    -1) + 2\sigma \sqrt{2(3 \sigma -1)} }{ 2 (1 - \sigma) } \ .
  \end{gather*}
  Substituting Notation \ref{notation:xsigmagamma} in the above inequality
  we finally obtain part \ref{item:uppercanonicalbound} of the Theorem. 

  We are going to prove Theorem \ref{item:uppercanonicalboundsmooth} now.
  Let $t = ( K_X + D ) \cdot C$, since $C$ is smooth 
  and $D$, $C$ meet transversally, $C^2 + ( K_X + D ) \cdot C = - 
  e_{C\setminus D}$, substituting the preceding inequalities in
  \eqref{equation:mainsecondary} we obtain 
  \begin{equation}
    \label{equation:polynomialsmooth}
    \begin{split}
      2 t^2 + \left[ 6 e_{C\setminus D} - 2 \left( 3 
	  e_{X\setminus D}- 
      \left( K_X + D \right)^2 \right) \right]
      t 
      - 2 e_{C\setminus D}  \left( 3 
      e_{X\setminus D}- \left( K_X + D \right)^2 \right) +
      \dfrac{9}{2} e_{C\setminus D}^2 \le 0\ .
    \end{split}
  \end{equation}
  As in the proof of the preceding Part of the Theorem, the leading
  coefficient of the above polynomial in $t$ is greater than or equal
  to zero, and its discriminant equals 
  \begin{equation*}
    \left( 3 e_{X\setminus D} - 
    \left( K_X + D \right)^2 \right) \left( -2 e_{C\setminus D} + 
    3 e_{X\setminus D} - 
    \left( K_X + D \right)^2 \right) \ .
  \end{equation*}
  If $C$ is not a smooth $D$-rational curve then $- e_{C\setminus D}\ge 0$,
  moreover by the log Bogomolov-Miyaoka-Yau inequality as in   
  Miyaoka \cite[Corollary 1.2]{miyaoka-maximal},
  the above discriminant is non negative and finally $t$ is bounded
  above by the largest root of the polynomial in
  \eqref{equation:polynomialsmooth}. A straightforward computation gives
  now \eqref{equation:canonicalsmooth}.
  
  The proof of \eqref{MichaelsEqn} follows exactly as per 
  \cite[Remark A, pg. 405]{miyaoka-orbibundle}
  but employs \cite[Theorem 0.1]{langer-orbifold}
  so as to have the Bogomolov-Miyaoka-Yau
  inequality in the generality necessary to deal with $D\cap C\neq \emptyset$,
  indeed after contracting $C$ we may obtain a quotient singularity on the
  boundary divisor. 
  Specifically the curve $C$ of \eqref{MichaelsEqn} has self intersection
  $C^2=-m$, so that on its contraction $(X_0,D_0)$ taking into account the
  contribution by the quotient singularity \footnote{See 
    \cite[p.~359]{langer-orbifold} for the
  computation of the relative orbifold Euler numbers.}, if any,
  \begin{equation}\label{Michael1}
    e_{X_0\setminus D_0} =\begin{cases}
      e_{X\setminus D} -2 + \frac{1}{m}\quad & \text{if}\,\, 
      C\cap D=\emptyset, \\
      e_{X\setminus D} -1 \quad &\text{otherwise}
    \end{cases}
  \end{equation}
  and similarly,
  \begin{equation}\label{Michael2}
    \left(K_{X_0} + D_0\right)^2 = \begin{cases}
      \left(K_X + D\right)^2 + \dfrac{(m-2)^2}{m}\quad & \text{if}\,\, 
      C\cap D=\emptyset, \\
      \left(K_X + D\right)^2 + \dfrac{(m-1)^2}{m}\quad & \text{otherwise}
    \end{cases}
  \end{equation}
  while the canonical degree is given by
  \begin{equation}\label{Michael3}
    t = (K_{X} + D) \cdot C =\begin{cases}
      m-2\quad & \text{if}\,\, C\cap D=\emptyset, \\
      m-1\quad & \text{otherwise}
    \end{cases}
  \end{equation}
  so that from \eqref{Michael1}-\eqref{Michael3} the orbifold
  Bogomolov-Miyaoka-Yau inequality,
  \begin{equation}\label{Michael4}
    \left(K_{X_0} + D_0\right)^2\leq 3 e_{X_0\setminus D_0}
  \end{equation}
  becomes in the notation of \eqref{Michael3}
  \begin{equation}\label{Michael5}
    3 e_{X\setminus D} - \left(K_X + D\right)^2 \geq
    \begin{cases}
      \frac{t^2}{t+2} - \frac{3}{t+2} +6 \quad &
      \text{if}\,\, C\cap 
      D=\emptyset, \\
      \frac{t^2}{t+1} + 3 \quad & \text{otherwise}\ .
    \end{cases}
  \end{equation}
  Observe that
  \begin{equation*}
    \dfrac{t^2}{t+2} - \dfrac{3}{t+2} +6\ge t+ 3 
  \end{equation*}
  for every $t\ge 0$, moreover since $ 3 e_{X\setminus D} - \left(K_X +
  D\right)^2   $ is an integer, we have 
  \begin{equation*}
    \Big \lceil  \dfrac{t^2}{t+1} + 3 \Big \rceil \le 3 e_{X\setminus D} - 
    \left(K_X +
    D\right)^2\ , 
  \end{equation*}
  but for $t\ge 0$: 
  \begin{equation*}
    \Big \lceil  \dfrac{t^2}{t+1} + 3 \Big \rceil = \Big \lceil
    \dfrac{t^2-1}{t+1} + \dfrac{1}{t+1}
    + 3 \Big \rceil = \Big \lceil t + 2 + \dfrac{1}{t+1} \Big \rceil = t + 3 
  \end{equation*}
   whence \eqref{MichaelsEqn} from \eqref{Michael5}.
\end{proof}

We conclude with
\begin{proof}[Proof of Corollary \ref{corollary:p1example}]
  First of all, observe that the degree of $K_{ \mathbb{P}^2} + D$ equals
  $(d_1+d_2 - 3)$, since  $d_2=6>3$ then
  $ K_{ \mathbb{P}^2} + D$ is nef and $ \left(K_{ \mathbb{P}^2} +
  D\right)^2> 0 $. Moreover,
  \begin{gather}
    D \cdot C = (d_1 + d_2)\ d > 1 \\
    e_{\mathbb{P}^2 \setminus D} = 3 + d_1(d_1 - 3 ) + d_2( d_2 -3) + d_1 d_2
    \\
    (K_{\mathbb{P}^2} + D)^2 = (d_1 + d_2)^2 - 6 (d_1 + d_2) + 9 
  \end{gather}
  and
  \begin{equation}
    \begin{split}
      (K_{ \mathbb{P}^2 } + D)^2 & - e_{\mathbb{P}^2 \setminus
      D} 
      = d_1 d_2 - 3 ( d_1 + d_ 2) + 6 \\
      &
      = \lambda d_2^2 - 3(\lambda + 1) d_2 + 6
      = \lambda d_2 \left( d_2 - 3(\lambda + 1) \right) + 6> 0
    \end{split} 
  \end{equation}
   then the hypotheses of Theorem \ref{item:uppercanonicalbound} are 
  satisfied. Expressing the quantities in Theorem 
  \ref{item:uppercanonicalbound} in terms of the given data we obtain:
  \begin{gather*}
    e_{\mathbb{P}^2 \setminus D} = (\lambda^2 + \lambda + 1) d_2^2 - 3
    (\lambda + 1) d_2 + 3 \le (\lambda^2 + \lambda + 1)d_2^2 \\
    ( K_{ \mathbb{P}^2} + D )^2 = (\lambda + 1 )^2 d_2^2
    -6(\lambda + 1) d_2 + 9 \le (\lambda + 1 )^2 d^2_2 \\
        3 e_{\mathbb{P}^2 \setminus D} - ( K_{ \mathbb{P}^2} + D )^2 
    = ( 2\lambda^2 + \lambda + 2) d_2^2 - 3
    (\lambda + 1 ) d_2 \le  ( 2\lambda^2 + \lambda + 2) d_2^2 \\
        -\frac{1}{2} e_{ C\setminus D } \le g - 1 +
    \frac{(d_1 + d_2)d}{2m} = g - 1 + \frac{(\lambda + 1) }{2m} \nu d_2^2
  \end{gather*}
  moreover
  \begin{gather*}
    ( K_{ \mathbb{P}^2} + D )\cdot C = (d_1+d_2-3) d \ge  \left(
    \lambda + \frac{1}{2} \right)\nu d_2^2 \\
    ( K_{ \mathbb{P}^2 } + D )^2  - e_{ \mathbb{P}^2 \setminus D } 
    = \lambda d_2^2 - 3( \lambda + 1) d_2 + 6 \ge
    \left( \frac{\lambda}{2} - \frac{1}{3} \right)d^2_2 \\
  \end{gather*}
  and if we fix $\lambda_0> \frac{2}{3}$, then for $\lambda \ge \lambda_0$ we
  have $\left( \frac{\lambda}{2} - \frac{1}{3} \right)d^2_2 > 0$. 
  We can now bound $A$, $B$ in \eqref{equation:uppercanonicalbound} in
  terms of $\lambda$ and $d_2$, indeed substituting the above expressions
  in the definition of $A$ and $B$, we get 
  \begin{equation*}
    A \le \mathrm{a}(\lambda),\quad B \le d_2^2\ \mathrm{b} (\lambda)
  \end{equation*}
  where
  \begin{equation*}
    \mathrm{a}(\lambda) = \dfrac{ 2(\lambda + 1 )^2 + (\lambda + 1
      ) \sqrt{  2
    (2 \lambda^2 + \lambda + 2)} }{ \left( \frac{\lambda}{2} -
    \frac{1}{3} \right) }
  \end{equation*}
  and
  \begin{equation*}
    \mathrm{b} (\lambda) = 
    \dfrac{(\lambda + 1)^2 
      (2 \lambda^2 + \lambda + 2) +
      (\lambda^2 + \lambda + 1) (\lambda + 1
      ) \sqrt{  2
    (2 \lambda^2 + \lambda + 2)} }{2 \left( \frac{\lambda}{2} -
    \frac{1}{3} \right) }\ .
  \end{equation*}
  Summing up, we can rewrite \eqref{equation:uppercanonicalbound} as 
  \begin{equation*}
    \left( \lambda + \frac{1}{2} \right)\nu d_2^2 
    \le \left( (\lambda + 1)d_2  -3\right) \nu d_2 \le 
    \mathrm{a}(\lambda) 
    \left( (g-1) + \frac{(\lambda + 1) d_2^2 }{2 m} \nu \right) + d_2^2
    \mathrm{b}(\lambda)
  \end{equation*}
  from which rearranging terms and dividing by $d_2^2$ we get:
  \begin{equation} \label{equation:boundni}
    \left[\left( \lambda + \frac{1}{2} \right) - \frac{(\lambda +
    1) \mathrm{a}(\lambda)}{2m} \right] \nu 
    \le \frac{\mathrm{a}(\lambda)}{d_2^2}
    (g-1)+
    \mathrm{b}(\lambda)\ .
  \end{equation}
  If $m > \frac{(\lambda + 1) \mathrm{a}(\lambda)}{\left( \lambda +
  \frac{1}{2} \right)}$ from \eqref{equation:boundni} we have:
  \begin{equation*}
    \left( \frac{\lambda_0}{2} - \frac{1}{3} \right) \left(
    \frac{\lambda_0}{2} + \frac{1}{4} \right) \nu \le
    \frac{4}{9}g + 22\ , 
  \end{equation*}
  and this concludes the proof of the Corollary if we set
  $\mathrm{h} = \frac{4}{9}$ and $\mathrm{k} = 22$, indeed 
  \begin{equation*}
    \dfrac{(\lambda + 1) \mathrm{a}(\lambda)}{\left( \lambda +
    \frac{1}{2} \right)} \le \dfrac{50}{ \left( \frac{\lambda}{2} -
      \frac{1}{3} \right)}  
  \end{equation*}
  if $\frac{2}{3}< \lambda \le 1$.
\end{proof}


\end{document}